\documentclass[10pt]{article}
\usepackage{amssymb}
\usepackage{graphicx}
\usepackage{xcolor} 
\usepackage{tensor}
\usepackage{fullpage} 
\usepackage{amsmath}
\usepackage{amsthm}
\usepackage{verbatim}
\usepackage{hyperref}
\DeclareFontFamily{U}{mathx}{\hyphenchar\font45}
\DeclareFontShape{U}{mathx}{m}{n}{
	<5> <6> <7> <8> <9> <10>
	<10.95> <12> <14.4> <17.28> <20.74> <24.88>
	mathx10
}{}
\DeclareSymbolFont{mathx}{U}{mathx}{m}{n}
\DeclareFontSubstitution{U}{mathx}{m}{n}
\DeclareMathAccent{\widecheck}{0}{mathx}{"71}
\DeclareMathAccent{\wideparen}{0}{mathx}{"75}

\usepackage{enumitem}
\setlist[enumerate]{leftmargin=1.5em}
\setlist[itemize]{leftmargin=1.5em}

\definecolor{green}{rgb}{0,0.8,0} 

\newtheorem{maintheorem}{Theorem}

\newtheorem{theorem}{Theorem}[section]

\newtheorem{lemma}[theorem]{Lemma}

\theoremstyle{definition}

\theoremstyle{remark}
\newtheorem{remark}[theorem]{Remark}

\numberwithin{equation}{section}

\makeatletter
\newcommand{\nrm}{\@ifstar{\nrmb}{\nrmi}}
\newcommand{\nrmi}[1]{\Vert{#1}\Vert}
\newcommand{\nrmb}[1]{\left\Vert{#1}\right\Vert}
\newcommand{\abs}{\@ifstar{\absb}{\absi}}
\newcommand{\absi}[1]{\vert{#1}\vert}
\newcommand{\absb}[1]{\left\vert{#1}\right\vert}
\newcommand{\brk}{\@ifstar{\brkb}{\brki}}
\newcommand{\brki}[1]{\langle{#1}\rangle}
\newcommand{\brkb}[1]{\left\langle{#1}\right\rangle}
\newcommand{\set}{\@ifstar{\setb}{\seti}}
\newcommand{\seti}[1]{\{#1\}}
\newcommand{\setb}[1]{\left\{ #1\right\}}
\makeatother

\newcommand{\nnrm}[1]{{\vert\kern-0.25ex\vert\kern-0.25ex\vert #1 
    \vert\kern-0.25ex\vert\kern-0.25ex\vert}}

\newcommand{\VERT}[1]{{\left\vert\kern-0.25ex\left\vert\kern-0.25ex\left\vert #1 
    \right\vert\kern-0.25ex\right\vert\kern-0.25ex\right\vert}}

\let\Re\relax
\DeclareMathOperator{\Re}{Re}

\newcommand{\aleq}{\lesssim}

\newcommand{\lap}{\Delta}

\newcommand{\ud}{\mathrm{d}}
\newcommand{\rd}{\partial}
\newcommand{\nb}{\nabla}

\newcommand{\alp}{\alpha}
\newcommand{\bt}{\beta}
\newcommand{\gmm}{\gamma}
\newcommand{\Gmm}{\Gamma}
\newcommand{\dlt}{\delta}

\newcommand{\eps}{\epsilon}

\newcommand{\kpp}{\kappa}

\newcommand{\Lmb}{\Lambda}

\newcommand{\tht}{\theta}

\newcommand{\omg}{\omega}
\newcommand{\Omg}{\Omega}

\newcommand{\bbR}{\mathbb R}

\newcommand{\bbT}{\mathbb T}

\newcommand{\ttht}{\tilde{\tht}}				

\newcommand{\normif}[1]{{\left\Vert #1 \right\Vert}_{L^{\infty}}}
\newcommand{\normb}[1]{{\left\Vert #1 \right\Vert}_{L^2}}

\newcommand{\normbr}[1]{{\left\Vert #1 \right\Vert}_{L^2\left (\mathbb{R}^2\right)}}
\newcommand{\normhs}[1]{{\left\Vert #1 \right\Vert}_{H^s}}

\newcommand{\normhsps}[1]{{\left\Vert #1 \right\Vert}_{H^{s+1}}}
\newcommand{\normhsr}[1]{{\left\Vert #1 \right\Vert}_{H^{s}(\mathbb{R}^2)}}

\begin{document}

\title{Well-posedness for Ohkitani model and long-time existence for surface quasi-geostrophic equations}
\author{Dongho Chae\thanks{Department of Mathematics, Chung-Ang University University.  E-mail: dchae@cau.ac.kr} \and In-Jee Jeong\thanks{Department of Mathematical Sciences and RIM, Seoul National University and School of Mathematics, Korea Institute for Advanced Study.  E-mail: injee\_j@snu.ac.kr}\and Jungkyoung Na\thanks{Department of Mathematics, Brown University.  E-mail: jungkyoung\_na@brown.edu}\and Sung-Jin Oh\thanks{Department of Mathematics, UC Berkeley and School of Mathematics, Korea Institute for Advanced Study. E-mail: sjoh@math.berkeley.edu}}

\date{\today}



\maketitle


\begin{abstract}
	We consider the Cauchy problem for the logarithmically singular surface quasi-geostrophic (SQG) equation, introduced by Ohkitani, \begin{equation*}
	\begin{split}
	\rd_t \tht - \nb^\perp \log(10+(-\lap)^{\frac12})\tht \cdot \nb \tht = 0 ,
	\end{split}
	\end{equation*} and establish local existence and uniqueness of smooth solutions in the scale of Sobolev spaces with exponent decreasing with time.  {Such a decrease of the Sobolev exponent is necessary, as we have shown in the companion paper \cite{CJO1} that the problem is strongly ill-posed in any fixed Sobolev spaces.} The time dependence of the Sobolev exponent can be removed when there is a dissipation term strictly stronger than log. These results improve wellposedness statements by Chae, Constantin, C\'{o}rdoba, Gancedo, and Wu in \cite{CCCGW}. 
	This well-posedness result can be applied to describe the long-time dynamics of the $\dlt$-SQG equations, defined by 
	\begin{equation*}
	\begin{split}
	\rd_t \tht + \nb^\perp (10+(-\lap)^{\frac12})^{-\dlt}\tht \cdot \nb \tht = 0,
	\end{split}
	\end{equation*} for all sufficiently small $\dlt>0$ depending on the size of the initial data. For the same range of $\dlt$, we establish global well-posedness of smooth solutions to the dissipative SQG equations. 
\end{abstract}

\section{Introduction}

\subsection{Main results}

In this paper, we are concerned with the \textit{logarithmically singular} surface quasi-geostrophic (SQG) equations: 
\begin{equation}\label{eq:sqg-log}
\left\{
\begin{aligned}
	&\rd_{t} \tht + u \cdot \nb \tht = 0, \\
	&u = -\nb^{\perp} \log(10 + \Lmb) \tht,
\end{aligned}
\right.
\end{equation}
where $\nb^{\perp} = (-\rd_{x_2}, \rd_{x_1})^{\top}$, $\Lmb= (-\lap)^{\frac12}$, $\tht(t,\cdot): \Omg\rightarrow \bbR$ and $u(t,\cdot):\Omg\rightarrow \bbR^2$ with $\Omg=\bbT^2$, $\bbR^2$, or $\bbT\times\bbR$. Note that \eqref{eq:sqg-log} is an inviscid nonlinear transport equation and the velocity is strictly more singular than the advected scalar. The system \eqref{eq:sqg-log} has been introduced by Ohkitani (\cite{Oh1,Oh2}) and was referred to as the \textit{Ohkitani model} later in \cite{CCCGW}. Our first main result gives a local well-posedness result in the scale of Sobolev spaces with an exponent \textit{decreasing} with time.
\begin{maintheorem}\label{thm:wp-log}
	For any $s_0>4$ and $\tht_0\in H^{s_0}(\Omg)$, there exist $T = T(s_0,\nrm{\tht_0}_{H^{s_0}})>0$, a continuous decreasing function $s(t)>4$ with $s(0)=s_0$ defined in $t\in[0,T]$, and a solution $\tht\in L^\infty([0,T];H^{s(t)}(\Omg))$ to \eqref{eq:sqg-log} with initial data $\tht_0$ satisfying \begin{equation*}
		\begin{split}
			\nrm{\tht(t,\cdot)}_{H^{s(t)}(\Omg)} \le C\nrm{\tht_0}_{H^{s_0}(\Omg)}.
		\end{split}
	\end{equation*} The solution is unique in the class $L^\infty([0,T];H^4(\Omg))$. 
\end{maintheorem} 

\begin{remark}
	One may extend the above well-posedness result to the cases where the multiplier $\log(10 + \Lmb)$ in the Biot--Savart law $u = \nb^{\perp} \log(10 + \Lmb) \tht$ is replaced with less singular radial multipliers, i.e. $\gmm=\gmm(|\xi|)$ satisfying $\gmm(|\xi|)\lesssim \log(10+|\xi|)$ together with a few reasonable assumptions, e.g. $|\rd^\alp_{\xi}\gmm (\xi)| \lesssim_\alp \brk{\xi}^{-|\alp|}|\gmm(\xi)|$ for any multi-index $\alp$. Some natural examples are $\log(10-\lap)$, $\log^{\beta}(10+\Lmb)$ ($\beta<1$), $\log^\alp(10 + \log(10+\Lmb))$ ($\alp>0$) etc.  
\end{remark}

As we shall explain below, the logarithmic Biot--Savart law featured in \eqref{eq:sqg-log} represents the natural borderline in which we can expect a well-posedness theory in Sobolev spaces. Interestingly, Ohkitani arrived at the same model in \cite{Oh1,Oh2} by investigating the collective behavior of the solutions for the generalized SQG equations:
\begin{equation}\label{eq:sqg-delta}
	\left\{
	\begin{aligned}
		&\rd_{t} \tht^{\dlt} + u^{\dlt} \cdot \nb \tht^{\dlt} = 0, \\
		&u^{\dlt} = \nb^{\perp} (10+\Lmb)^{-\dlt} \tht^{\dlt},
	\end{aligned}
	\right.
\end{equation}
when $\dlt>0$ is a parameter. He numerically observed that the sequence $\{ \tht^\dlt \}_{\dlt>0}$ is convergent after a time rescaling. Formally, this can be seen as follows: introducing the new time variable $\tau = \dlt t$, \eqref{eq:sqg-delta} can be written as \begin{equation*}
	\begin{split}
		\rd_\tau \tht^{\dlt} + \nb^\perp \left( \frac{(10+\Lmb)^{-\dlt}-1}{\dlt} \right) \tht^\dlt \cdot \nb \tht^\dlt = 0, 
	\end{split}
\end{equation*} and then from the Taylor expansion, \begin{equation}\label{lim: ohki}
\begin{split}
	\frac{(10+\Lmb)^{-\dlt}-1}{\dlt} \longrightarrow -\log(10+\Lmb)
\end{split}
\end{equation} as $\dlt\to0$. That is, the time-rescaled Biot--Savart law for \eqref{eq:sqg-delta} converges to that for \eqref{eq:sqg-log}. This seems to suggest that the solutions to \eqref{eq:sqg-delta} may converge to those to \eqref{eq:sqg-log}. We make this precise in the following result. 
\begin{maintheorem}\label{thm:long-time dynamics}
	Given any $\dlt\in(0,1)$ and initial data $\tht_0 \in H^{s}(\Omega)$ with $s>5$, there exists a unique solution $\tht^{\dlt}\in L^\infty([0,C^*\dlt^{-1}];H^{s}(\Omega))$ to \eqref{eq:sqg-delta} for some  constant $C^*>0$ depending only on $s$ and $\nrmb{\tht_0}_{H^s}$.
	Moreover, there exists a continuous decreasing function $s(t)>4$ defined in $t\in[0,C^*\dlt^{-1}]$ satisfying
	\begin{equation}\label{beh. of sol.}
		\nrmb{\tht^{\dlt}(t,\cdot)-\tht(\dlt t,\cdot)}_{H^{s(t)}(\Omega)} \le \dlt,
	\end{equation}
	where $\tht$ is the unique local solution to \eqref{eq:sqg-log} with the same initial data $\tht_0$.
\end{maintheorem}
\begin{remark}
	Observe that the maximal existence time $T_{max}^{\dlt}$ of the smooth solution of \eqref{eq:sqg-delta} satisfies $T_{max}^{\dlt}\gtrsim \dlt^{-1} \rightarrow \infty$ as $\dlt \rightarrow 0^+$. Furthermore, \eqref{beh. of sol.} shows that the long-time behavior of the solution sequence for \eqref{eq:sqg-delta} is given by the solution of the Ohkitani model. 
\end{remark}

\medskip

Moreover,  {the convergence \eqref{lim: ohki} yields the \textit{global} existence of smooth solutions for the following dissipative counterpart of \eqref{eq:sqg-delta}:}
\begin{equation}\label{eq:diss-sqg-delta}
	\left\{
	\begin{aligned}
		&\rd_{t} \tht^{\dlt} + u^{\dlt} \cdot \nb \tht^{\dlt} + \kpp  \Psi\tht^{\dlt} = 0, \\
		&u^{\dlt} = \nb^{\perp} (10+\Lmb)^{-\dlt} \tht^{\dlt},
	\end{aligned}
	\right.
\end{equation}
when $\dlt>0$ is sufficiently small. Here $\Psi$  {can be any super-logarithmic Fourier multiplier in the sense that its symbol $\psi$ satisfies the lower bound}
\begin{equation}\label{eq:diss-sqg-cond}
	\begin{split}
		\psi(|\xi|) \ge \log(10+|\xi|),\quad |\xi|>\Xi_0 
	\end{split}
\end{equation} for some $\Xi_0 > 0$. 
\begin{maintheorem}\label{thm:gwp}
    Given any $\kpp>0$ and initial data $\tht_0 \in H^{s}(\Omega)$ with  {$s>4$}, there exists a unique global solution $\tht^{\dlt} \in L^{\infty}([0,\infty);H^{s}(\Omega))$ to \eqref{eq:diss-sqg-delta} for all $\dlt\in(0,\dlt^*)$, where $\dlt^*\in(0,1)$ is a constant depending only on $\kpp$, $s$, and $\nrmb{\tht_0}_{H^s}$, and satisfies
    \begin{equation*}
        \lim_{\nrmb{\tht_0}_{H^s}\rightarrow \infty}\dlt^*=0.
    \end{equation*}
\end{maintheorem}

\medskip

\noindent  {Lastly, we} consider the following dissipative counterpart of \eqref{eq:sqg-log}: \begin{equation}\label{eq:sqg-log-diss}
	\left\{
	\begin{aligned}
		&\rd_{t} \tht + u \cdot \nb \tht + \kpp  \log^{\beta} (10 + \Lmb)\tht = 0, \\
		&u = -\nb^{\perp} \log(10 + \Lmb) \tht.
	\end{aligned}
	\right.
\end{equation}
The following result shows that when there is a dissipative term stronger than log, there is no need to consider time-dependent Sobolev spaces. 
\begin{maintheorem}\label{thm:wp-log-diss}
	For $\beta>1$, the dissipative system \eqref{eq:sqg-log-diss} is locally well-posed in $L^\infty_{t}H^{s}_{x}$ for any $s>4$, That is, given any initial data $\tht_0 \in H^{s}(\Omg)$, $\kpp>0$, there exist $T>0$ and a unique solution $\tht\in L^\infty([0,T];H^{s}(\Omg))$ to \eqref{eq:sqg-log-diss} with initial data $\tht_0$. 
\end{maintheorem}

\begin{remark} This result extends to more general systems of the form \begin{equation}\label{eq:diss-2}
		\left\{
		\begin{aligned}
					&\rd_{t} \tht + u \cdot \nb \tht + \kpp \Upsilon \tht = 0, \\
			&u = -\nb^{\perp} \log(10 + \Lmb) \tht,
		\end{aligned}
		\right.
	\end{equation}  where the dissipation is  given by a Fourier multiplier $\Upsilon$ with symbol $\upsilon$ satisfying  \begin{equation}\label{eq:diss-cond}
	\begin{split}
		\frac{\upsilon(|\xi|)}{\log(10+|\xi|)} \ge \frac{C_s}{\kappa}\nrm{\tht_0}_{H^s(\Omg)},\quad |\xi|>\Xi_0 
	\end{split}
\end{equation} for some $\Xi_0 > 0$ and  $C_s>0$.  {Moreover, the requirement $\bt>1$ is sharp in view of the illposedness result from \cite{CJO1} in the case $\bt\le1$.}
\end{remark}

\subsection{Previous works} \label{subsec:prev-work}

\noindent \textbf{Generalized SQG equations.} The generalized SQG equation (gSQG) takes the form \begin{equation}\label{eq:gSQG}
	\begin{split}
		&\rd_t \tht + u\cdot\nb\tht = 0, \\
		& u = \nb^\perp \Gmm \tht,
	\end{split}
\end{equation} in a two-dimensional physical domain without a boundary, where $\Gmm$ is some Fourier multiplier. The special cases $\Gmm = \Lmb^{-2}$ and $\Lmb^{-1}$ correspond respectively to the two-dimensional incompressible Euler equation and the (usual) SQG equation, derived in \cite{CMT1,CMT2}. Furthermore, the equation in the case $\Gmm = \Lmb$ governs the dynamics of the so-called electron MHD equation under certain symmetric ansatz (see \cite{CJO1,LM}). It seems that the generalized model \eqref{eq:gSQG} was first considered in \cite{CCW3}, and subsequently several authors investigated the question of well-posedness for various multipliers $\Gmm$ (\cite{CCW2,CCCGW,ElgindiA,FV,FGSV}). In the particular case of $\Gmm = \Lmb^{\beta}$, standard energy estimates give local well-posedness for $\beta \le -1$, and an additional cancellation arising from the even symmetry of $\Gmm$ is necessary to obtain local well-posedness in the case $\beta\le 0$ (\cite{CCCGW}).\footnote{This even symmetry is essential; in the works \cite{FV,FGSV}, \textit{ill-posedness} was established for odd $\Gmm$ with scaling $|\gmm(\xi)| \sim |\xi|^{\beta}$ with $-1<\beta<0$.} The latter argument can be adapted to multipliers $\Gmm$ with even symbols $\gmm$ satisfying $|\gmm(\xi)|\lesssim 1$, together with a few natural regularity assumptions on the derivatives of $\gmm$.  It turns out that this boundedness assumption is crucial for local well-posedness (\cite{CJO1}), and for this reason, we shall refer to $\Gmm$ as a \textit{singular} multiplier if its symbol $\gmm$ satisfies $|\gmm(\xi)|\rightarrow\infty$ as $|\xi|\rightarrow\infty$.

\medskip

\noindent \textbf{Global well-posedness for dissipative SQG equations.}
The dissipative SQG equations takes the form
\begin{equation}\label{eq:sqg-diss}
	\left\{
	\begin{aligned}
		&\rd_{t} \tht + u \cdot \nb \tht + \kpp(-\lap)^{\alp}\tht = 0, \\
		&u = \nb^{\perp}\Lmb^{-1} \tht
	\end{aligned}
	\right.
\end{equation}
for $\kpp>0$ and $\alp\in(0,1)$.
The cases $\alpha>\frac{1}{2}$, $\alpha=\frac{1}{2}$, and $\alpha<\frac{1}{2}$ are referred to as subcritical, critical, and supercritical, respectively. Global existence and uniqueness of smooth solutions corresponding to arbitrary smooth initial data were established for the subcritical regime in \cite{CW99}, and for the critical regime in \cite{crit1, crit2, crit3, crit4} (see also \cite{DKSV11} for the logarithmically supercritical case). In the supercritical regime, the question of global regularity was settled only for sufficiently small initial data (\cite{CL03, CC04, W04, CMZ07}). Global well-posedness or finite-time blow-up for arbitrary initial data in the supercritical case remains an outstanding open problem. In this direction, the authors of \cite{CV16} proved that the supremum of sizes of initial data leading to global well-posedness goes to infinity as $ {\alp \rightarrow \frac{1}{2}}$. Comparing this result with our Theorem \ref{thm:gwp}, both of them sacrifice some quantities to obtain global regularity. Indeed, to absorb the size of large initial data, the order of dissipation increases to the order corresponding to the critical case in \cite{CV16} while the order of velocity goes to the order corresponding to the trivial steady state in Theorem \ref{thm:gwp}. However, recalling \eqref{eq:diss-sqg-delta} and \eqref{eq:diss-sqg-cond}, Theorem \ref{thm:gwp} only requires logarithmic dissipation $\log(10+\Lmb)$, which is much weaker than $(-\lap)^{\alp}$ with $\alp\in(0,\frac12)$ in \eqref{eq:sqg-diss}.

\medskip

\noindent \textbf{Wellposedness in the Ohkitani model}. Chae, Constantin, Cordoba, Gancedo, and Wu have established a local existence and uniqueness result for the logarithmically singular SQG equations in \cite{CCCGW}. Precisely, in \cite[Theorem 1.3]{CCCGW}, the authors proved that the system \begin{equation}\label{eq:sqg-log-diss-lap}
	\left\{
	\begin{aligned}
		&\rd_{t} \tht + u \cdot \nb \tht + \kpp(-\lap)^{\alp}\tht = 0, \\
		&u = \nb^{\perp}( \log(10 - \lap))^{\mu} \tht
	\end{aligned}
	\right.
\end{equation} for any $\kpp, \mu, \alpha>0$ is locally well-posed in $L^\infty([0,T];H^{4}(\bbR^2))$. On the other hand, our Theorem \ref{thm:wp-log} says that when $\mu\le 1$, there is no need for a dissipation term for local well-posedness. Furthermore, Theorem \ref{thm:wp-log-diss} states that a logarithmic dissipation suffices; although we omit the proof, Theorem \ref{thm:wp-log-diss} can be generalized to any $\mu>0$, by simply replacing the condition \eqref{eq:diss-cond} with \begin{equation}\label{eq:diss-cond-gen}
	\begin{split}
		\frac{\upsilon(|\xi|)}{(\log(10+|\xi|))^{\mu}} \ge \frac{C_s}{\kappa}\nrm{\tht_0}_{H^s}.
	\end{split}
\end{equation}

\medskip

\noindent \textbf{Convergence to the Ohkitani model.} As we shall explain below, the key idea in the well-posedness of \eqref{eq:sqg-log} is the use of losing estimate in the scale of Sobolev spaces. Interestingly, the same strategy can be applied to the family of generalized SQG equations \eqref{eq:sqg-delta} for $\dlt>0$, this time giving \textit{long-time} existence for smooth solutions as $\dlt\to0$. Furthermore, the solutions to \eqref{eq:sqg-delta} as $\dlt\to0$ converges to the solution of \eqref{eq:sqg-log}, after a time rescaling. This confirms the expectations of Ohkitani \cite{Oh1,Oh2}. We have been also inspired by the \textit{incompressible Euler limit} of \eqref{eq:sqg-delta} (namely, $\dlt\to2$) recently considered in \cite{YZJ}. This limit is interesting as well since while the two-dimensional Euler equation is known to be globally well-posed. The authors of \cite{YZJ} actually show the long-time existence of solutions to \eqref{eq:sqg-delta} as $\dlt\to2$.

\medskip

\noindent \textbf{Illposedness in the Ohkitani model.} The above main theorems nicely complement the \textit{illposedness} results from our recent work \cite{CJO1}; among other things, one can prove a \textit{nonexistence} result for \eqref{eq:sqg-log} and \eqref{eq:sqg-log-diss}  {in a fixed Sobolev space $H^{s}$ with $s$ arbitrarily large,} in the sense that there exists a smooth initial data $\tht_0$ for which there cannot exist a corresponding solution in $L^\infty([0,T];H^{s})$ for $s>3$ with any $T>0$. The dissipative version requires a condition of the form (roughly) \begin{equation}\label{eq:diss-cond-illposed}
	\begin{split}
		\frac{\upsilon(|\xi|)}{\log(10+|\xi|)} \ll 1,
	\end{split}
\end{equation} see \cite{CJO1} for the precise statements. This shows that  {\emph{a decrease in the Sobolev exponent} (as in Theorem~\ref{thm:wp-log}) \emph{is unavoidable if one insists in having wellposedness in the $H^{s}$-spaces}.} The paper \cite{CJO1} actually proves nonexistence for a large class of singular even multipliers $\Gmm$.

\subsection{Ideas of the proof}
To understand how such a losing estimate in the scale of time-dependent Sobolev spaces is possible  {even when the time-independent estimate fails}, it is instructive to make an analogy with the well-known Cauchy--Kovalevskaya theorem, which gives local well-posedness of the Cauchy problem in the analytic class when the loss of derivative is less than or equal to one\footnote{We do not expect the analytic regularity to propagate in time for the gSQG equation \eqref{eq:gSQG} with \textit{singular} $\Gmm$; that is, Cauchy--Kovalevskaya theorem fails. Indeed, one may prove some analytic (and even Gevrey) illposedness results for linear and nonlinear SQG equations near quadratically degenerate shear steady states, closely following the arguments of \cite{CJO1,JO1}. Note that formally the loss of derivative in the gSQG equation with singular $\Gmm$ is always strictly larger than one.}. 

To be concrete, the key point in the proof of Cauchy--Kovalevskaya theorem for the model equation $\rd_t b = \Lmb b$ on $\bbR$ with initial data $b(t=0)=b_0$ is that if one considers the time-dependent multiplier $m(t,\xi) := \exp( (c_0 - Ct)|\xi|), $ (with $c_0>0$ depending on the radius of analyticity of $b_0$) then we have \begin{equation*}
\begin{split}
	\frac{\ud}{\ud t} \nrm{ m(t,\xi) \widehat{b}(\xi) }_{L^2}^2 + C \nrm{ |\xi|^{\frac12} m(t,\xi) \widehat{b}(\xi) }_{L^2}^2 = \int  m^2(t,\xi)\overline{\widehat{b}(\xi)} \rd_t\widehat{b}(\xi) \,\ud \xi 
\end{split}
\end{equation*} so that for $C>0$ large, we can absorb the right hand side into the positive term on the left hand side. Now, as a reasonable linear model of the logarithmically singular SQG equation  {whose initial value problem is illposed in a \emph{fixed} Sobolev space \cite{CJO1, JO1}}, we consider the degenerate dispersive equation \begin{equation}\label{eq:deg-dis}
	\begin{split}
		\rd_t b = x\rd_x \ln(10+\Lmb )b + \frac{1}{2}\ln(10+\Lmb )b
	\end{split}
\end{equation} on $\bbR$. More generally, we can consider \begin{equation}\label{eq:deg-dis-gen}
	\begin{split}
		\rd_t b = x\rd_x \Gmm(\Lmb )b + \frac{1}{2} \Gmm(\Lmb )b , 
	\end{split}
\end{equation} where $\Gmm$ is a multiplier with the symbol $\gmm=\gmm(|\xi|)\ge 0$ \textit{singular} in the sense that $\gmm(|\xi|)\rightarrow \infty$ as $|\xi|\rightarrow\infty$. When $\gmm$ is \textit{not} singular, one may prove a Cauchy-Kovaleskaya type theorem under some natural assumptions on the derivatives of $\gmm$. The presence of last terms in \eqref{eq:deg-dis} and \eqref{eq:deg-dis-gen} is not essential; they were inserted to just guarantee the conservation \begin{equation*}
	\begin{split}
		\frac{1}{2}\frac{\ud}{\ud t} \nrm{\gmm^{\frac{1}{2}}(\Lmb )b}_{L^2}^2= -\frac{1}{2} \nrm{\gmm^{\frac{1}{2}}(\Lmb )b}_{L^2}^2 + \frac{1}{2} \nrm{\gmm^{\frac{1}{2}}(\Lmb )b}_{L^2}^2 = 0
	\end{split}
\end{equation*} which models conservation of $\nrm{\Gmm^{\frac12}\tht}_{L^{2}}$ in \eqref{eq:gSQG}. To explore possible function spaces in which \eqref{eq:deg-dis-gen} is well-posed, we consider a time-dependent multiplier $m(t,\xi)$ (positive and increasing with $|\xi|$). Then, we compute that from \begin{equation*}
	\begin{split}
		\rd_t \widehat{b}(\xi) = -\rd_{\xi} (\xi \gmm(|\xi|) \widehat{b}) + \frac{1}{2} \gmm(|\xi|)\widehat{b}, 
	\end{split}
\end{equation*} \begin{equation*}
	\begin{split}
		\frac{1}{2}\frac{\ud}{\ud t} \int m^2(t,\xi)|\widehat{b}(t,\xi)|^2 = \int m \dot{m} |\widehat{b}|^2 + \Re \int  m^2  \,\overline{\widehat{b} }\left(-\rd_{\xi} (\xi \gmm(|\xi|) \widehat{b}) + \frac{1}{2} \gmm(|\xi|)\widehat{b} \right)
	\end{split}
\end{equation*} and \begin{equation*}
	\begin{split}
		-\Re \int  m^2  \,\overline{\widehat{b} } \rd_{\xi} (\xi \gmm(|\xi|) \widehat{b})& = \Re \int \rd_{\xi}( m^2 \, \overline{\widehat{b} } ) \xi \gmm(|\xi|) \widehat{b}   = \Re \int (2m\rd_{\xi} m  \overline{\widehat{b}} + m^2 \rd_{\xi}  \overline{\widehat{b}}  ) \xi \gmm(|\xi|)\widehat{b} \\
		& = \Re \int 2m\rd_\xi m \xi \gmm(|\xi|) |\widehat{b}|^2  -\frac{1}{2} \Re \int \rd_{\xi} ( m^2 \gmm(|\xi|)\xi ) |\widehat{b}|^2 
	\end{split}
\end{equation*} that \begin{equation*}
	\begin{split}
		\frac{1}{2}\frac{\ud}{\ud t} \int m^2(t,\xi)|\widehat{b}(t,\xi)|^2 = \int m(\xi)\left( \dot{m} + \xi \gmm(|\xi|)\rd_\xi m - m\xi\rd_\xi \gmm(|\xi|) \right)|\widehat{b}|^2 \le 0
	\end{split}
\end{equation*} if $m$ is chosen that \begin{equation}\label{eq:m-ineq}
	\begin{split}
		\dot{m} + \xi \gmm(|\xi|)\rd_\xi m - m\xi\rd_\xi \gmm(|\xi|) \le 0. 
	\end{split}
\end{equation} Define the characteristics $\dot{\Xi} = \Xi \gmm(|\Xi|), \Xi(0)=\xi $ so that the above is equivalent to \begin{equation*}
	\begin{split}
		\frac{\ud}{\ud t} m(t,\Xi(t,\xi)) \le \Xi \gmm'(|\Xi|) m(t,\Xi(t,\xi)). 
	\end{split}
\end{equation*} Hence if the characteristics $\Xi$ are well-defined for $O(1)$-time uniformly for all $\xi$, then we can take \begin{equation*}
	\begin{split}
		m(t,\Xi(t,\xi)) = m_0(\xi) \exp\left( \int_0^t  \Xi(s,\xi) \gmm'(|\Xi(s,\xi)|) ds  \right).
	\end{split}
\end{equation*} In the case of $\gmm(|\xi|) = \ln(1+|\xi|)$, we have that \eqref{eq:m-ineq} is given by $\dot{m} + \xi \ln(1+|\xi|) \rd_\xi m - m \le 0, $ and  for large $|\xi|$, \begin{equation*}
	\begin{split}
		\Xi(t,\xi) \simeq \xi^{\exp(t)}, \qquad 		m(t,\xi^{\exp(t)}) \simeq m_0(\xi) \exp(t). 
	\end{split}
\end{equation*} Hence, in the Sobolev case where $m_0(\xi) \simeq |\xi|^s$, $m(t,\xi) \simeq  |\xi|^{s\exp(-t)} \exp(t)$. That is, a losing estimate in the scale of $H^{s(t)}$-spaces is available.\footnote{In the Gevrey case where $m_0(\xi) \simeq \exp(|\xi|^{\frac{1}{\sigma}})$ for some $\sigma>0$, we have that $m(t,\xi) \simeq \exp( |\xi|^{\frac{1}{e^t\sigma }} )$. We see that a losing estimate in the Gevrey class can be obtained, but unlike the usual case, the index $\sigma$ should be sacrificed and not the radius of Gevrey regularity.}

In parallel with the above argument, if one considers the time-dependent multiplier of the form \begin{equation*}
\begin{split}
	\widetilde{m}(t,\xi) := |\xi|^{s_{0}-Ct}
\end{split}
\end{equation*}(with $s_0>0$ depending on  $\tht_0$ and $C>0$ large) for a solution $\tht$ of \eqref{eq:sqg-log}, then we have \begin{equation*}
\begin{split}
	\frac{\ud}{\ud t} \nrm{ \widetilde{m}(t,\xi) \widehat{\tht}(t,\xi) }_{L^2_\xi}^2 + C \nrm{ (\log |\xi|)^{\frac12} \widetilde{m}(t,\xi) \widehat{\tht}(t,\xi) }_{L^2_\xi}^2 =  \int  \widetilde{m}^2(t,\xi)\overline{\widehat{\tht}(t,\xi)} \rd_t\widehat{\tht}(t,\xi) \,\ud \xi 
\end{split}
\end{equation*} so that if the loss of regularity in the right hand side is \textit{sub-logarithmic}, then we may deduce  that \begin{equation*}
\begin{split}
	\frac{\ud}{\ud t} \nrm{ \widetilde{m}(t,\xi) \widehat{\tht}(t,\xi) }_{L^2_\xi}^2 \le 0
\end{split}
\end{equation*} which results in an a priori estimate. This type of observation (that is, handling a logarithmic loss of derivative with a losing estimate in the scale of Sobolev spaces) is not new. For instance, one may consider the task of solving the 2D incompressible Euler equation (SQG with $\Gmm=(-\lap)^{-1}$) with Sobolev critical initial data:\begin{equation*}
\begin{split}
	\rd_t\omg + \nb^\perp (-\lap)^{-1}\omg\cdot\nb\omg = 0. 
\end{split}
\end{equation*} When $\omg_0\in W^{s,p}$ with $sp=2$ and $1<p<\infty$, we \textit{do not} have in general $\nb \nb^\perp (-\lap)^{-1}\omg \in L^\infty$, which is necessary to propagate $W^{s,p}$-regularity. Indeed, Bourgain and Li in  \cite{BL1} has shown that (see \cite{EJ,JY2,KJ} for simplified proofs) there exist initial data $\omg_0\in W^{s,p}\cap L^\infty \cap L^{1}$ whose unique $L^\infty_{t}(L^\infty \cap L^{1})$ solution\footnote{For the 2D Euler equations with $\omg_0\in L^\infty \cap L^{1}$ initial data, Yudovich's theorem gives a unique global in time solution $\omg \in L^\infty_t(L^\infty \cap L^{1})$ (\cite{Yudovich1963a}).} \textit{instantaneously} escapes $W^{s,p}$. That is, 2D Euler is strongly illposed in critical Sobolev spaces. On the other hand, it is not difficult to see that such an illposedness behavior is logarithmic in nature; one way to quantify it is to consider the logarithmically regularized Euler equation (\cite{CCW3}), defined by \begin{equation*}
\begin{split}
	\rd_t\omg + \nb^\perp (-\lap)^{-1} (\log(10-\lap))^{-\gmm}\omg\cdot\nb\omg = 0
\end{split}
\end{equation*}  for some $\gmm>0$. It has been shown in \cite{ChWu} that the above is locally wellposed in $W^{s,p}$ for $\gmm>\frac12$. (On the other hand, there is still strong illposedness for $\gmm\le\frac12$; see \cite{Kwon}.) Since the loss is only logarithmic, one can consider a losing estimate in the scale of Sobolev spaces; indeed, with $H^1 \cap L^\infty \cap L^{1}$ initial data, it can be shown that the unique $L^\infty_t(L^\infty \cap L^{1})$ solution at time $t$ belongs to $H^{s(t)} \cap W^{1,p(t)}$ where $s,p$ are continuously decreasing functions of time with $s(0)=p(0)=2$ (\cite{EJ,Jthesis,Hung,J-JMFM}).

\subsection{Open problems} 

Given local well-posedness, the natural question is to ask  whether local solutions could blow up in finite time or not. For the inviscid equation \eqref{eq:sqg-log}, one may also ask whether there exists a solution with Sobolev regularity actually decreasing continuously with time. We plan to address these issues in the future. 

In general, not much is known regarding the long-time dynamics and the possibility of singularity formation for the gSQG equations. While the 2D Euler equation (\eqref{eq:gSQG} with $\Gmm=(-\lap)^{-1}$) is globally well-posed, this global regularity result so far only extends to models \textit{slightly} more singular than 2D Euler (by loglog in \cite{CCW3}). Several constructions of smooth solutions to gSQG equations with Sobolev norms growing with time have been obtained in \cite{Den,Den2,KN,HeKi,Z,Xu,CJ-Lamb,DEJ}. Finite-time singularity formation for a class of patch-type solutions touching the boundary was established for a range of gSQG equations in \cite{KRYZ,KYZ}.

\section{Proofs}\label{sec:wp-log}

 {We give the proofs of Theorems \ref{thm:wp-log} and \ref{thm:wp-log-diss} in \S \ref{subsec:wp-log}, and Theorems \ref{thm:long-time dynamics} and \ref{thm:gwp} in \S \ref{subsec:rela.}.} In the proofs, we shall assume that $\Omg=\bbR^2$, but the case of $\bbR\times\bbT$ and $\bbT^2$ can be treated in the same manner. 

\subsection{Local well-posedness for \eqref{eq:sqg-log} and \eqref{eq:sqg-log-diss}}\label{subsec:wp-log}
We need an elementary lemma.
\begin{lemma}\label{lem:elementary}
	For any $s\ge 3$, we have the inequality \begin{equation*}
	\begin{split}
	\left| |\xi|^s - |\eta|^s - |\xi-\eta|^s - s(\xi-\eta)\cdot\eta |\xi-\eta|^{s-2} \right| \le C_s ( |\eta|^2|\xi-\eta|^{s-2} + |\xi-\eta||\eta|^{s-1} )
	\end{split}
	\end{equation*} for any $\xi,\eta\in\bbR^2$. Here $C_s>0$ is a constant depending only on $s$. 
\end{lemma}
 {We note that $C_{s}$ could be taken to be independent of $s$ if we restrict $s$ to a closed interval.}
\begin{proof}
	To establish the inequality, we introduce \begin{equation*}
	\begin{split}
	A(\rho) = (1-\rho)\xi+\rho \eta, \quad B(\rho) = (1-\rho)(\xi-\eta)
	\end{split}
	\end{equation*} so that we have \begin{equation*}
	\begin{split}
	|\xi|^s - |\eta|^s = \int_0^1 \rd_\rho | A(\rho) |^s \, \ud\rho  = -s(\eta-\xi) \cdot  \int_0^1 |A(\rho)|^{s-2} A(\rho)\, \ud\rho ,
	\end{split}
	\end{equation*}\begin{equation*}
	\begin{split}
	|\xi|^s - |\eta|^s - |\xi-\eta|^s  = -s(\eta-\xi) \int_0^1 \left(|A(\rho)|^{s-2} A(\rho) - |B(\rho)|^{s-2} B(\rho) \right)\, \ud\rho.
	\end{split}
	\end{equation*} We write \begin{equation*}
	\begin{split}
	|A(\rho)|^{s-2} A(\rho) - |B(\rho)|^{s-2} B(\rho) = |B(\rho)|^{s-2} (A(\rho) -  B(\rho)) + (|A(\rho)|^{s-2}- |B(\rho)|^{s-2}) A(\rho) = I + II 
	\end{split}
	\end{equation*} and observe the following formulas:  \begin{equation*}
	\begin{split}
	-s(\eta-\xi)\int_0^1 I \, \ud\rho = \frac{s}{s-1} (\xi-\eta)\cdot \eta |\xi-\eta|^{s-2},
	\end{split}
	\end{equation*}\begin{equation*}
	\begin{split}
	|A(\rho)|^{s-2}- |B(\rho)|^{s-2} = \int_0^1 \rd_\sigma |E(\rho,\sigma)|^{s-2}\, \ud\sigma = -(s-2)\eta \cdot \int_0^1 |E|^{s-4}E \, \ud\sigma ,
	\end{split}
	\end{equation*} and \begin{equation*}
	\begin{split}
	-s(\eta-\xi) \int_0^1 II \,\ud\rho = -s(s-2)(\eta-\xi)\cdot\int_0^1\int_0^1 A(\rho) |E|^{s-4} \eta \cdot E \,\ud\sigma\,\ud\rho ,
	\end{split}
	\end{equation*} where $E (\rho,\sigma) = (1-\rho)\xi + (\rho-\sigma)\eta.$ We may therefore write \begin{equation*}
	\begin{split}
	 |\xi|^s - |\eta|^s - |\xi-\eta|^s - s(\xi-\eta)\cdot\eta |\xi-\eta|^{s-2}  & = -s(\eta-\xi) \int_0^1 II \,\ud\rho - \frac{s(s-2)}{s-1} (\xi-\eta) \cdot \eta |\xi-\eta|^{s-2} \\
	& = s(s-2)(\xi-\eta) \left[ \int_0^1 \int_0^1  A|E|^{s-4}\eta\cdot E - \frac{1}{s-1}\eta |\xi-\eta|^{s-2} \ud\sigma \ud\rho \right] .
	\end{split}
	\end{equation*} Noticing that \begin{equation*}
	\begin{split}
	\left| A|E|^{s-4}\eta\cdot E - \frac{1}{s-1}\eta |\xi-\eta|^{s-2} \right| \le C_s |\eta|^2 \left( |\xi-\eta|^{s-3} + |\eta|^{s-3} \right)
	\end{split}
	\end{equation*} holds, the proof is complete. 
\end{proof}

\begin{proof}[Proof of Theorem \ref{thm:wp-log}]
	We consider the inviscid case.  We shall perform an $H^{s}$ estimate, with $s = s(t)>4$ depending on time. For convenience, we re-define the $H^s$ norm by \begin{equation}\label{redef-H^s}
\begin{split}
\nrm{\tht}_{H^s(\bbR^2)}^2 = \int_{\bbR^2} (10+|\xi|)^{2s} |\widehat{\tht}(\xi)|^2 \, \ud\xi.
\end{split}
\end{equation} Then, we compute that \begin{equation*}
\begin{split}
\frac{1}{2}\frac{\ud}{\ud t} \nrm{\tht}_{ {H}^{s(t)}}^2 & =  \int \dot{s}(t)  (10+|\xi|)^{2s(t)} \log(10 + |\xi|) |\widehat{\tht}(\xi)|^2  \, \ud\xi \\
&\qquad  - \Re \int ( 10+|\xi|)^{2s(t)} \overline{\widehat{\tht} }(\xi) \int  \left( \widehat{u}(\xi-\eta) \cdot i\eta  \widehat{\tht}(\eta)    \right) \,\ud\eta  \, \ud\xi. 
\end{split}
\end{equation*}  Neglecting the constant 10 for now, we shall estimate the following quantity with $s = s(t)$: \begin{equation}\label{eq: decom-thmA}
\begin{split}
\Re \int  |\xi|^{ s} \overline{\widehat{\tht} }(\xi) \int  |\xi|^{ s}\left( \widehat{u}(\xi-\eta) \cdot i\eta  \widehat{\tht}(\eta)    \right) \,\ud\eta  \, \ud\xi  =: I + II + III + IV,
\end{split}
\end{equation} where the four terms are given by the decomposition \begin{equation}\label{eq: decomp-xi}
\begin{split}
|\xi|^{ s} = |\eta|^s + |\xi-\eta|^s + s(\xi-\eta)\cdot\eta |\xi-\eta|^{s-2} + R(\eta,\xi) . 
\end{split}
\end{equation} That is, we define \begin{equation*}
\begin{split}
I = \Re \int  |\xi|^{ s} \overline{\widehat{\tht} }(\xi) \int  |\eta|^{ s}\left( \widehat{u}(\xi-\eta) \cdot i\eta  \widehat{\tht}(\eta)    \right)  \,\ud\eta  \, \ud\xi ,
\end{split}
\end{equation*}\begin{equation*}
\begin{split}
II = \Re \int  |\xi|^{ s} \overline{\widehat{\tht} }(\xi) \int  |\xi-\eta|^{ s}\left( \widehat{u}(\xi-\eta) \cdot i\eta  \widehat{\tht}(\eta)    \right)  \,\ud\eta  \, \ud\xi ,
\end{split}
\end{equation*}\begin{equation*}
\begin{split}
III = \Re \int  |\xi|^{ s} \overline{\widehat{\tht} }(\xi) \int  s(\xi-\eta)\cdot\eta |\xi-\eta|^{s-2}\left( \widehat{u}(\xi-\eta) \cdot i\eta  \widehat{\tht}(\eta)    \right)  \,\ud\eta  \, \ud\xi ,
\end{split}
\end{equation*} and \begin{equation*}
\begin{split}
IV = \Re \int  |\xi|^{ s} \overline{\widehat{\tht} }(\xi) \int  R(\xi,\eta)\left( \widehat{u}(\xi-\eta) \cdot i\eta  \widehat{\tht}(\eta)  \right) \,\ud\eta  \, \ud\xi .
\end{split}
\end{equation*} To begin with,  {we claim that $I = 0$. By symmetrizing\footnote{This argument is equivalent to using the anti-symmetry of $u \cdot \nb$.} with respect to the change of variables $(\xi, \eta) = (\eta', \xi')$, and also using $\overline{\widehat{u}}(\xi-\eta) = \widehat{u}(-\xi+\eta)$ (since each component of $u$ is real-valued), we see that
\begin{equation*}
I = \frac{1}{2} \Re \int  |\xi|^{ s} \overline{\widehat{\tht} }(\xi) \int  |\eta|^{ s}\left( \widehat{u}(\xi-\eta) \cdot i(\eta-\xi)  \widehat{\tht}(\eta)    \right)  \,\ud\eta  \, \ud\xi,
\end{equation*}
which vanishes since $u$ is divergence-free (hence $\widehat{u}(\xi-\eta) \cdot i(\eta-\xi) = 0$).} Next, we claim that \begin{equation*}
\begin{split}
\left| IV \right| \le C_s \nrm{\tht}_{H^s}^3
\end{split}
\end{equation*} for $s>4$. To see this, we recall the definition of $u$ and the form of $R$ from Lemma \ref{lem:elementary}: with Young's convolution inequality, \begin{equation*}
\begin{split}
IV &\le \iint |\xi|^{ s} |\overline{\widehat{\tht} }(\xi)| ( |\eta|^2|\xi-\eta|^{s-2} + |\xi-\eta||\eta|^{s-1} ) \log(10+|\xi-\eta|)|\xi-\eta| |\widehat{\tht}(\xi-\eta)|  |\eta||\widehat{\tht}(\eta)|  \,\ud\eta  \, \ud\xi  \\
&\le \iint |\xi|^{ s} |\overline{\widehat{\tht} }(\xi)|  |\xi-\eta|^{s-1}  \log(10+|\xi-\eta|) |\widehat{\tht}(\xi-\eta)|  |\eta|^3|\widehat{\tht}(\eta)|  \,\ud\eta  \, \ud\xi \\
&\qquad + \iint |\xi|^{ s} |\overline{\widehat{\tht} }(\xi)|  \log(10+|\xi-\eta|)|\xi-\eta|^2 |\widehat{\tht}(\xi-\eta)|  |\eta|^s|\widehat{\tht}(\eta)|   \,\ud\eta  \, \ud\xi  \\
& \le C\nrm{  {\brk{\xi}^3} \widehat{\tht}(\xi) }_{L^1_\xi} \nrm{\tht}_{H^s}^2 \le C_s\nrm{\tht}_{H^s}^3.
\end{split}
\end{equation*} (Observe that $s>4$ is required in the last inequality to apply Sobolev embedding in two dimensions.)  It remains to handle the terms $II$ and $III$. To estimate $II$, we may rewrite\begin{equation*}
\begin{split}
II& = -\Re \iint |\xi|^s \overline{\widehat{\tht} }(\xi)\log^{\frac{1}{2}}(10+|\xi-\eta|)\\
&\phantom{= -\Re \iint}
\times ( \log^{\frac{1}{2}}(10+|\xi-\eta|)-\log^{\frac{1}{2}}(10+|\xi|)) \widehat{\tht}(\xi-\eta)|\xi-\eta|^{ s}  (\xi-\eta)^\perp \cdot \eta   \widehat{\tht}(\eta)    \,\ud\eta  \, \ud\xi .
\end{split}
\end{equation*}  {Here, we used }\begin{equation*}
\begin{split}
 \Re \iint |\xi|^s \overline{\widehat{\tht} }(\xi) \log^{\frac{1}{2}}(10+|\xi|)\log^{\frac{1}{2}}(10+|\xi-\eta|) \widehat{\tht}(\xi-\eta)|\xi-\eta|^{ s}  (\xi-\eta)^\perp \cdot \eta   \widehat{\tht}(\eta)  \,\ud\eta  \, \ud\xi = 0 ,
\end{split}
\end{equation*}  {which follows by symmetrizing\footnote{This argument is equivalent to using the anti-symmetry of $\log^{\frac{1}{2}}(10+\Lmb) \nb \tht \cdot \nb^{\perp} \log^{\frac{1}{2}}(10+\Lmb)$.} with respect to the change of variables $(\xi, \eta) = (\xi'-\eta', -\eta')$ and using $\overline{\widehat{\tht}}(\eta) = \widehat{\tht}(-\eta)$ (since $\tht$ is real-valued)}. We write \begin{equation*}
\begin{split}
\log^{\frac{1}{2}}(10+|\xi-\eta|)-\log^{\frac{1}{2}}(10+|\xi|) = \frac{\log(10+|\xi-\eta|)-\log(10+|\xi|)}{\log^{\frac{1}{2}}(10+|\xi-\eta|)+\log^{\frac{1}{2}}(10+|\xi|)}
\end{split}
\end{equation*} and deduce a rough estimate \begin{equation*}
\begin{split}
\left| \log^{\frac{1}{2}}(10+|\xi-\eta|)-\log^{\frac{1}{2}}(10+|\xi|) \right| \le C\left| \log \frac{10+|\xi-\eta|}{10+|\xi|} \right| \le C|\eta|\left(\frac{1}{10+|\xi|} + \frac{1}{10+|\xi-\eta|} \right) . 
\end{split}
\end{equation*} With this in hand, we can bound \begin{equation*}
\begin{split}
\left| II  \right| &\le C \iint  |\xi|^s |\widehat{\tht}(\xi)| \log^{\frac{1}{2}} (10+|\xi-\eta|) |\xi-\eta|^s |\widehat{\tht}(\xi-\eta)| |\eta|^2|\widehat{\tht}(\eta)|   \,\ud\eta  \, \ud\xi  \\
&\le C \nrm{|\eta|^2\widehat{\tht}(\eta)}_{L^1_\eta} \int \log(10+|\xi|) |\xi|^{2s} |\widehat{\tht}(\xi)|^2\, \ud\xi .
\end{split}
\end{equation*}
 {Here, we used $(\xi - \eta)^{\perp} \cdot \eta = \xi^{\perp} \cdot \eta$ to estimate $((10+|\xi|)^{-1} + (10+|\xi-\eta|)^{-1}) \abs{(\xi-\eta)^{\perp} \cdot \eta} \aleq \abs{\eta}$.}
 For $III$, it would be convenient to make a change of variables $\eta \mapsto \mu$ by $\mu = \xi-\eta$. Then, using the definition of $u$, we have that \begin{equation*}
\begin{split}
III &=-\Re \int |\xi|^{ s} \overline{\widehat{\tht} }(\xi) \int s (\mu\cdot(\xi-\mu)) (\mu^\perp \cdot (\xi-\mu)) |\mu|^{s-2} \log(10+|\mu|) \widehat{\tht}(\mu)  \widehat{\tht}(\xi-\mu)   \,\ud\mu  \, \ud\xi  \\
&= -\Re \iint |\xi|^2 \log(10+|\mu|)  G(\xi,\mu)\overline{\widehat{\tht} (\xi)}{\widehat{\tht} }(\mu)\widehat{\tht}(\xi-\mu)   \,\ud\mu  \, \ud\xi
\end{split}
\end{equation*} where $G (\xi,\mu) := s|\xi|^{s-2}|\mu|^{s-2}  (\mu\cdot(\xi-\mu)) (\mu^\perp \cdot (\xi-\mu)).$ It encourages us to consider the symmetrization \begin{equation*}
\begin{split}
III' := -\Re \iint |\xi||\mu| \log^{\frac{1}{2}}(10+|\mu|)\log^{\frac{1}{2}}(10+|\xi|)  G(\xi,\mu)\overline{\widehat{\tht} (\xi)}{\widehat{\tht} }(\mu)\widehat{\tht}(\xi-\mu)  \,\ud\mu  \, \ud\xi .
\end{split}
\end{equation*} With Young's convolution inequality  {and the simple identity $\mu^{\perp} \cdot (\xi - \mu) = (\mu - \xi)^{\perp} \cdot \xi$, we see that} \begin{equation*}
\begin{split}
\left| III' \right| &\le C \nrm{|\eta|^2\widehat{\tht}(\eta)}_{L^1_\eta} \int \log(10+|\xi|) |\xi|^{2s} |\widehat{\tht}(\xi)|^2\, \ud\xi .
\end{split}
\end{equation*} Next, $III-III'$ is given by \begin{equation*}
\begin{split}
-\Re\iint K(\xi,\mu)  G(\xi,\mu)\overline{\widehat{\tht} (\xi)}{\widehat{\tht} }(\mu)\widehat{\tht}(\xi-\mu)   \,\ud\mu  \, \ud\xi
\end{split}
\end{equation*} where \begin{equation*}
\begin{split}
|K(\xi,\mu)| &= |\xi| \log(10+|\mu|) \log^{\frac{1}{2}}(10+|\xi|) \left| \frac{|\xi|}{\log^{\frac{1}{2}}(10+|\xi|)} - \frac{|\mu|}{\log^{\frac{1}{2}}(10+|\mu|)}\right|\\
&\le C|\xi||\xi-\mu| \log(10+|\mu|) \log^{\frac{1}{2}}(10+|\xi|) \\
&\le C|\xi||\xi-\mu|( \log(10+|\xi|) + \log(10+|\xi-\mu|) ) \log^{\frac{1}{2}}(10+|\xi|).
\end{split}
\end{equation*} Observing $\nrm{\log(10+|\eta|)|\eta|^3\widehat{\tht}(\eta)}_{L^1_\eta} \le C \nrm{ {\tht}}_{H^s},$ we obtain from the above bound of $K$ that \begin{equation*}
\begin{split}
|III-III'|\le C\nrm{\tht}_{H^s}^3. 
\end{split}
\end{equation*} Together with the previous bound for $III'$, we conclude \begin{equation*}
\begin{split}
\left| III \right| &\le C \nrm{\tht}_{H^s} \int \log(10+|\xi|) |\xi|^{2s} |\widehat{\tht}(\xi)|^2\, \ud\xi   + C_s\nrm{\tht}_{H^{s }}^3. 
\end{split}
\end{equation*} Combining the estimates for $I$, $II$, $III$, and $IV$, we have\begin{equation}\label{eq:apriori-main}
\begin{split}
&\left| \Re \int  |\xi|^{ s} \overline{\widehat{\tht} }(\xi) \int  |\xi|^{ s}\left( \widehat{u}(\xi-\eta) \cdot i\eta  \widehat{\tht}(\eta)    \right)  \,\ud\eta  \, \ud\xi \right|  \le C_s \nrm{\tht}_{H^s} \int \log(10+|\xi|) |\xi|^{2s} |\widehat{\tht}(\xi)|^2   \, \ud\xi+ C_s\nrm{\tht}_{H^{s }}^3. 
\end{split}
\end{equation} We now claim that the expression\begin{equation*}
\begin{split}
V := \Re \iint (( 10+|\xi|)^{2s} - |\xi|^{2s}) \overline{\widehat{\tht} }(\xi)  \left( \widehat{u}(\xi-\eta) \cdot i\eta  \widehat{\tht}(\eta)    \right)  \,\ud\eta  \, \ud\xi
\end{split}
\end{equation*} in absolute value is bounded by the right hand side of \eqref{eq:apriori-main}. We only sketch the argument; setting\begin{equation*}
\begin{split}
Z(\xi) := \sqrt{( 10+|\xi|)^{2s} - |\xi|^{2s}}, 
\end{split}
\end{equation*} we consider \begin{equation*}
\begin{split}
V' := \Re \iint Z(\xi)Z(\xi-\eta)  \overline{\widehat{\tht} (\xi)}  \left( \widehat{u}(\xi-\eta) \cdot i\eta  \widehat{\tht}(\eta)    \right) \,\ud\eta  \, \ud\xi
\end{split}
\end{equation*} There is a highest order cancellation in $V'$, which can be obtained by symmetrizing the log in $|\xi|$ and $|\xi-\eta|$. Furthermore, the highest order term in $|\xi-\eta|$ cancels in the expression $V-V'$, giving rise to a factor of $|\eta|^3$ in the kernel. We omit the details. 

Therefore, we have arrived at the following: \begin{equation*}
\begin{split}
\frac{1}{2}  \frac{\ud}{\ud t}  \nrm{\tht}_{ {H}^{s }}^2 \le   \left(\dot{s} + C_s \nrm{\tht}_{H^s}\right) \int \log(10+|\xi|) |\xi|^{2s} |\widehat{\tht}(\xi)|^2 \, \ud\xi+ C_s\nrm{\tht}_{H^{s }}^3. 
\end{split}
\end{equation*} We set $F(t) = \nrm{\tht}_{H^{s(t)}}$ and \textit{define} $s(t)$ to be a smooth function of time such that the system of differential inequalities\footnote{ {We note that the dependence of $C_{s}$ on $s$ is minor; it can be assumed to be independent of $s$ if we work on a closed interval of values of $s$.}} \begin{equation*}
\begin{split}
&\dot{s} + C_s F \le 0 , \qquad \dot{F} \le C_s F^2,
\end{split}
\end{equation*} are satisfied for some small interval of time. (One can simply take $s = s_0 - Mt$ for some large $M>0$.) Then, this allows us to prove $ F(t) \le CF_0$ for a possibly smaller interval of time. Hence, we have proved the a priori estimate $\nrm{\tht(t,\cdot)}_{H^{s(t)}} \le C\nrm{\tht_0}_{H^{s_0}}$ in $t\in[0,T]$ for some $T = T(s_0, \nrm{\tht_0}_{H^{s_0}} )>0$, as long as $s_0>4$. 

Proving existence and uniqueness of the solution satisfying this a priori estimate is straightforward and we shall only sketch the argument. For the proof of existence, one can consider the dissipative system with a Laplacian, namely \begin{equation*}
\begin{split}
\rd_t\tht+u\cdot\nb\tht =\kappa\lap \tht .
\end{split}
\end{equation*} There exists a unique local in time smooth solution $\tht^{(\kappa)}$ with initial data $\tht_0\in H^{s_0}$ for any $\kappa>0$ (\cite{CCCGW}) which clearly satisfies the $H^{s(t)}$ a priori estimate above. Therefore, there exists $T = T(s_0, \nrm{\tht_0}_{H^{s_0}} )>0$ such that the solutions  $\tht^{(\kappa)}$ are uniformly bounded in $L^\infty([0,T];H^{4+\varepsilon})$ for some $\varepsilon>0$. Furthermore, since  $\tht^{(\kappa)}$ are uniformly bounded in $L^\infty([0,T];L^\infty)$ as well. An interpolation gives that there exists (by taking a subsequence if necessary) a strong limit $\tht^{(\kappa)}\to \tht $ in $L^\infty([0,T];H^4)$ as $\kappa\rightarrow0$. It is straightforward to verify that $\tht$ is a solution to the inviscid equation with initial data $\tht_0$. To prove uniqueness, we assume that there exist two solutions $\tht$ and $\widetilde{\tht}$ to \eqref{eq:sqg-log} with the same initial data $\tht_0$, belonging to $L^\infty([0,T];H^4)$ for some $T>0$. We set $\phi = \tht-\widetilde{\tht}$ and compute, for some $M>0$ to be chosen later, \begin{equation*}
\begin{split}
&\frac{1}{2} \frac{\ud}{\ud t} \nrm{ (10+\Lmb)^{-M t} \phi }_{L^2}^2 + M\nrm{ (10+\Lmb)^{-M t} \log^{\frac{1}{2}}(10+\Lmb) \phi }_{L^2}^2 \\
&\qquad = -\brk{ \nb^\perp\log(10+\Lmb) \tht \cdot\nb\phi,  (10+\Lmb)^{-2M t}\phi} + \brk{ \nb^\perp \log(10+ \Lmb) \phi \cdot\nb\widetilde{\tht},  (10+\Lmb)^{-2M t}\phi }\\
&\qquad =: I + II. 
\end{split}
\end{equation*} Note that \begin{equation*}
\begin{split}
I &= - \brk{ (10+\Lmb)^{- M t}(\nb^\perp\log(10+\Lmb) \tht \cdot\nb\phi),  (10+\Lmb)^{- M t}\phi}  \\
& =  - \brk{ [(10+\Lmb)^{- M t} , \nb^\perp\log(10+\Lmb) \tht \cdot\nb ]\phi ,  (10+\Lmb)^{- M t}\phi}  \\
\end{split}
\end{equation*} and with a commutator estimate for $[(10+\Lmb)^{- M t} , \nb^\perp\log(10+\Lmb) \tht \cdot\nb ]$, \begin{equation*}
\begin{split}
\left| I \right| \le C\nrm{\tht}_{H^4} \nrm{ (10+\Lmb)^{- M t}\phi }_{L^2}^2.
\end{split}
\end{equation*} Similarly, for $II$ we can derive \begin{equation*}
\begin{split}
\left| II \right| \le C\nrm{\widetilde{\tht}}_{H^4} \nrm{ (10+\Lmb)^{-M t} \log^{\frac{1}{2}}(10+\Lmb) \phi }_{L^2}^2 . 
\end{split}
\end{equation*} Taking $M = 2C\nrm{\widetilde{\tht}}_{H^4} $ gives \begin{equation*}
\begin{split}
\frac{\ud}{\ud t} \nrm{ (10+\Lmb)^{-M t} \phi }_{L^2}^2 \le C\nrm{\tht}_{H^4} \nrm{ (10+\Lmb)^{- M t}\phi }_{L^2}^2,
\end{split}
\end{equation*} or \begin{equation*}
\begin{split}
\nrm{ (10+\Lmb)^{- M t}\phi(t) }_{L^2} \le \exp(C\nrm{\tht}_{H^4}t)\nrm{ \phi_0 }_{L^2} = 0.
\end{split}
\end{equation*} This shows $\phi(t) \equiv 0$, which finishes the proof of uniqueness. \end{proof}

\begin{remark}
	This proof does not say that $C^\infty$-smooth data remains $C^\infty$-smooth for some interval of time. It is only guaranteed that for any $k$, there exists some time interval $[0,t_k]$ with $t_{k}>0$ depending on $k$ such that the solution is $C^k$-smooth in space for $t \in [0,t_k]$. 
\end{remark}

\begin{proof}[Proof of Theorem \ref{thm:wp-log-diss}]
	To establish local well-posedness in the dissipative case, one can derive an estimate of the form \begin{equation*}
\begin{split}
\frac{1}{2} \frac{\ud}{\ud t} \nrm{\tht}_{ {H}^{s }}^2 + \kappa \int \upsilon(|\xi|) (10+|\xi|)^{2s}  |\widehat{\tht}(\xi)|^2 \,\ud\xi   \le  C_s \nrm{\tht}_{H^s} \int \log(10+|\xi|) |\xi|^{2s} |\widehat{\tht}(\xi)|^2 \,\ud\xi + C_s\nrm{\tht}_{H^{s }}^3
\end{split}
\end{equation*} by arguing as in the proof of Theorem \ref{thm:wp-log} above. Under the assumption \eqref{eq:diss-cond} on $\upsilon$, the first term on the right hand side can be absorbed to the left hand side, giving \begin{equation*}
\begin{split}
\frac{1}{2} \frac{\ud}{\ud t} \nrm{\tht}_{ {H}^{s }}^2 + \frac{\kappa}{2} \int \upsilon(|\xi|) (10+|\xi|)^{2s}  |\widehat{\tht}(\xi)|^2 \,\ud\xi  \le   C_s\nrm{\tht}_{H^{s }}^3. 
\end{split}
\end{equation*} This gives the a priori estimate $\nrm{\tht(t)}_{H^s}\le C_{s}\nrm{\tht_0}_{H^s}$ for $0\le t\le T$ with $T = T(\nrm{\tht_0}_{H^s})>0$. We omit the proof of existence and uniqueness.
\end{proof}

\subsection{The relation between solutions of \eqref{eq:sqg-delta} and \eqref{eq:sqg-log}}\label{subsec:rela.}
This subsection is devoted to  {the proofs of Theorem \ref{thm:long-time dynamics} and  {\ref{thm:gwp}}}. 
Hereafter, we denote the identity operator by $I$.
Before starting the proof, we present some lemmas. The first one is the well-known Kato--Ponce commutator estimate.
\begin{lemma}[\cite{KP88}]\label{lem:K-P}
For any $\eps>0$ and $s>0$, there exists a constant $C=C_{\eps,s}>0$ depending only on $\eps$ and $s$  satisfying
\begin{equation*}
    \normbr{\left[\Lmb^s, f\right]g}
    \le C \left(\nrm{f}_{H^{2+\eps}(\bbR^2)}\normbr{\Lmb^{s-1}g}
     + \normbr{\Lmb^{s}f}\nrm{g}_{H^{1+\eps}(\bbR^2)}\right).
\end{equation*}
\end{lemma}

The following estimates are concerned with velocities of \eqref{eq:sqg-delta} and \eqref{eq:sqg-log}.
\begin{lemma}\label{lem:Tay}
For any $\eps>0$, $0<\dlt<1$, and $s\ge0$, there exists a constant $C=C_{\eps}>0$ depending only on $\eps$ satisfying
\begin{subequations}
\begin{align}
    \normhsr{\nabla^{\perp} \left(\dlt\log(10 + \Lmb) -I + (10+\Lmb)^{-\dlt}\right)f} &\le C \dlt^2 \nrm{f}_{H^{s+1+\eps}(\bbR^2)},\label{eq:Tay-1}\\
    \normhsr{\nabla^{\perp} \left( (10+\Lmb)^{-\dlt}-I\right)f} &\le C \dlt \nrm{f}_{H^{s+1+\eps}(\bbR^2)}.\label{eq:Tay-2}
\end{align}
\end{subequations}
\end{lemma}
\begin{proof}
Plancherel's theorem gives us that
\begin{equation}\label{eq:Tay-3}
\begin{split}
    &\normhs{\nabla^{\perp} \left(\dlt\log(10 + \Lmb)-I+(10+\Lmb)^{-\dlt}\right)f}\\
    &\qquad= \normb{(1+|\xi|^2)^{\frac{s}{2}}|\xi|\left|\dlt\log(10 + |\xi|) - 1 + (10+|\xi|)^{-\dlt}\right|\hat{f}(\xi)}.
\end{split}
\end{equation}
Applying the Taylor's theorem to $(10+|\xi|)^{-\dlt}$ at $0$, we have
\begin{equation*}
\begin{split}
    \left|\dlt\log(10 + |\xi|) - 1 + (10+|\xi|)^{-\dlt}\right| &\le \dlt^2 \left|\log(10 + |\xi|)\right|^2= \dlt^2 (10+|\xi|)^{\eps}(10+|\xi|)^{-\eps}\left|\log(10 + |\xi|)\right|^2 \\
    & \le C\dlt^2 (10+|\xi|)^{\eps},
\end{split}
\end{equation*}
where in the last line, we used the fact that the maximum of $h(x):=x^{-\eps}(\log x)^2$ with $x \ge1$ is $\left(\frac{2}{\eps e}\right)^2$.
Plugging this into \eqref{eq:Tay-3}, we obtain \eqref{eq:Tay-1}.
With a similar argument, we obtain the estimate:
\begin{equation*}
    \left|(10 + |\xi|)^{-\dlt} - 1 \right| \le \dlt \left|\log(10 + |\xi|)\right| \le C\dlt|\xi| (10+|\xi|)^{\eps},
\end{equation*}
which yields \eqref{eq:Tay-2}.
\end{proof}

The next one is another type of commutator estimate.
\begin{lemma}\label{lem:Comm}
For any $\eps>0$ and $s\ge0$, there exists a constant $C=C_{\eps}>0$ depending only on $\eps$ satisfying
\begin{align*}
     &\normbr{\left[\rd_j \left(I-(10+\Lmb)^{-\dlt}\right)^{\frac{1}{2}}, f  \right]\Lmb^{s}g} \le  C\dlt^{\frac{1}{2}}\nrmb{f}_{H^{2+\eps}(\bbR^2)}\normhsr{\log^{\frac{1}{2}}(10+\Lmb)g}\quad (j=1,2).
\end{align*}
\end{lemma}
\begin{proof}
We note that
\begin{equation*}
\begin{split}
    &\normb{\left[\rd_j\left(I-(10+\Lmb)^{-\dlt}\right)^{\frac{1}{2}}, f  \right]\Lmb^{s}g} \le {\left\Vert  \int \left|\Phi_j(\xi,\eta)\right||\hat{f}(\eta)||\xi-\eta|^{s}|\hat{g}(\xi-\eta)|\, d\eta \right\Vert}_{L^2_{\xi}},
\end{split}
\end{equation*}
where
\begin{equation*}
    \Phi_j(\xi,\eta):=\xi_j\left(1-(10+|\xi|)^{-\dlt}\right)^{\frac{1}{2}} - (\xi-\eta)_j \left(1-(10+|\xi-\eta|)^{-\dlt}\right)^{\frac{1}{2}}.
\end{equation*}
By the mean value theorem, we obtain $\rho \in (0,1)$ such that
\begin{equation*}
    \begin{split}
        \left|\Phi_j(\xi,\eta)\right|
        &\le C|\eta|
        \left(\left(1-(10+|\rho \xi + (1-\rho)(\xi-\eta)|)^{-\dlt}\right)^{\frac{1}{2}}
        \right.\\
        &\left. \qquad+ \dlt(10+|\rho \xi 
        + (1-\rho)(\xi-\eta)|)^{-\dlt}\left(1-(10+|\rho \xi + (1-\rho)(\xi-\eta)|)^{-\dlt}\right)^{-\frac{1}{2}}\right) \\
        &= C|\eta|
        \left(\left(1-(10+|(\xi-\eta) + \rho\eta|)^{-\dlt}\right)^{\frac{1}{2}}\right.\\
        &\left.\qquad
        + \dlt(10+|(\xi-\eta) 
        + \rho\eta|)^{-\frac{\dlt}{2}}\left((10+|(\xi-\eta)+\rho\eta|)^{\dlt}-1\right)^{-\frac{1}{2}}\right).
    \end{split}
\end{equation*}
Using
\begin{equation*}
\begin{split}
    \left(1-(10+|(\xi-\eta) + \rho\eta|)^{-\dlt}\right)^{\frac{1}{2}}
    &\le \dlt^\frac12 \log^{\frac12}(10+|(\xi-\eta) + \rho\eta|) \le C \dlt^\frac12 \left(\log^{\frac{1}{2}}(10+|\xi-\eta|) 
    +\log^{\frac{1}{2}}(10+|\eta|)\right),
\end{split}
\end{equation*} 
\begin{equation*}
    \begin{split}
        &\dlt(10+|(\xi-\eta) 
        + \rho\eta|)^{-\frac{\dlt}{2}}\left((10+|(\xi-\eta)+\rho\eta|)^{\dlt}-1\right)^{-\frac{1}{2}}
        \le \dlt^{\frac12},
    \end{split}
\end{equation*}
we have
\begin{equation*}
    \begin{split}
        \left|\Phi_j(\xi,\eta)\right|
        &\le C\dlt^\frac12 \log^{\frac{1}{2}}(10+|\xi-\eta|)|\eta|\left(1+\log^{\frac12}(10+|\eta|)\right).
    \end{split}
\end{equation*}
Hence, Young's convolution inequality gives
\begin{equation*}
\begin{split}
    &\normb{\left[\rd_{j}\log^{\frac{1}{2}}(10+\Lmb),f\right]\Lmb^{s}g} \\ 
    &\qquad\le C \dlt^{\frac{1}{2}}{\left\Vert  \int |\eta|\left(1+\log^{\frac12}(10+|\eta|)\right)|\hat{f}(\eta)||\xi-\eta|^{s}\log^{\frac{1}{2}}(10+|\xi-\eta|)|\hat{g}(\xi-\eta)| d\eta \right\Vert}_{L^2_{\xi}} \\
    &\qquad\le C \dlt^{\frac{1}{2}}\nrmb{|\eta|\left(1+\log^{\frac12}(10+|\eta|)\right)|\hat{f}(\eta)|}_{L^{1}_{\eta}}\normhs{\log^{\frac{1}{2}}(10+\Lmb)g} \le C\dlt^{\frac{1}{2}}\nrmb{f}_{H^{2+\eps}}\normhs{\log^{\frac{1}{2}}(10+\Lmb)g}.  \qedhere 
\end{split}
\end{equation*} 
\end{proof}

Finally, we present an elementary lemma from \cite{C86}. 
\begin{lemma}\label{lem:PC}
Let $T>0$, $G>0$ be given constants and let $F(t)$ be a nonnegative continuous function on $[0,T]$. Assume further that $\nu$ be a constant satisfying
\begin{equation} \label{PC-1}
   0<8\nu TG \int_0^{T}F(t)dt\le 1.
\end{equation}
Then every solution $y(t)\ge 0$ of
\begin{equation} \label{PC-2}
    \left\{
    \begin{aligned}
    &\frac{\ud}{\ud t}y(t) \le \nu F(t) + Gy^2(t), \\
    &y(0)=0,
    \end{aligned}
    \right.
\end{equation}
is uniformly bounded on $[0,T]$ and satisfies
\begin{equation} \label{PC-3}
    y(t) \le \min \left\{\frac{3}{2TG},12\nu \int_{0}^{T}F(t)dt\right\}.
\end{equation}
\end{lemma}

\begin{proof}[Proof of Theorem \ref{thm:long-time dynamics}]
By Theorem \ref{thm:wp-log}, $\tht_0\in H^s$ with $s>5$ guarantees a time $T^{\log}>0$ and a unique solution $\tht\in L^{\infty}([0,T^{\log}];H^{s_0+1+\eps})$ to \eqref{eq:sqg-log} for some $s_0>4$ and $\eps>0$. Since $\tht^{\dlt}\in L^{\infty}_t H^{s}_x$ with $s>3$ implies $\tht^{\dlt}\in L^{\infty}_t H^{\infty}_x$ (see \cite{CCCGW}), it suffices to show the existence of a continuous decreasing function $s(t)>4$ with $s(0)=s_0$ defined on $[0,C^*\dlt^{-1}]$ satisfying
\begin{equation*}
    \nrmb{\tht^{\dlt}(t)-\tht(\dlt t)}_{H^{s(t)}} \le \dlt,
\end{equation*}
where $C^*>0$ is a constant to be determined later.
For simplicity, we denote \begin{equation*}
    \ttht^{\dlt}(x,t):=\tht(x,\dlt t), \qquad \text{and} \qquad \tht^{D}(x,t):=\tht^{\dlt}(x,t)-\ttht^{\dlt}(x,t).
\end{equation*}
Noticing $
    \nabla^{\perp}(10+\Lmb)^{-\dlt}\tht^{\dlt} \cdot \nabla \tht^{\dlt}= \nabla^{\perp}\left((10+\Lmb)^{-\dlt}-I\right)\tht^{\dlt} \cdot \nabla \tht^{\dlt},$
we have that 
\begin{equation*}
    \rd_t \tht^{\dlt} + \nb^{\perp}\left((10+\Lmb)^{-\dlt}-I\right)\tht^{\dlt} \cdot \nabla \tht^{\dlt}=0.
\end{equation*}
Moreover, observing $\rd_t \ttht^{\dlt} -  \nb^{\perp} \dlt\log(10 + \Lmb) \ttht^{\dlt} \cdot \nb \ttht^{\dlt} = 0,$
we obtain the equation of $\tht^{D}$:
\begin{equation*}\label{eq:D}
\begin{split}
    \rd_{t} & \tht^{D} +\nabla^{\perp}\left((10+\Lmb)^{-\dlt}-I\right)\tht^{D} \cdot \nabla \ttht^{\dlt} 
    +\nabla^{\perp}\left((10+\Lmb)^{-\dlt}-I\right)\tht^{D} \cdot \nabla \tht^{D} \\
    &+ \nabla^{\perp}\left((10+\Lmb)^{-\dlt}-I\right)\ttht^{\dlt} \cdot \nabla \tht^{D} 
    +\nabla^{\perp} \left(\dlt\log(10 + \Lmb)  -I + (10+\Lmb)^{-\dlt}\right)\ttht^{\dlt} \cdot \nabla \ttht^{\dlt}=0.
\end{split}
\end{equation*}
Re-defining the $H^s$ norm by \eqref{redef-H^s}, we have
\begin{equation*}
    \left|\frac{1}{2}\frac{\ud}{\ud t} \nrmb{\tht^D}_{H^{s(t)}}^2
    - \dot{s}(t) \nrmb{\log^{\frac{1}{2}}(10+\Lmb)\tht^D}^2_{H^{s(t)}} \right| \le
    \left|\int \tht^D \rd_t\tht^D \right| + \left| \int \Lmb^{s(t)}\tht^D \Lmb^{s(t)}\rd_t\tht^D \right|. 
\end{equation*}
To begin with, we estimate $\int \tht^D \rd_t\tht^D$. We compute that
\begin{align*}
    \int \tht^D \rd_t\tht^D &=  - \int \tht^{D}\left(\nabla^{\perp}\left((10+\Lmb)^{-\dlt}-I\right)\tht^{D} \cdot \nabla \ttht^{\dlt}\right)  - \int \tht^{D}\left(\nabla^{\perp}\left((10+\Lmb)^{-\dlt}-I\right)\tht^{D} \cdot \nabla \tht^{D}\right) \\
    &\quad  - \int \tht^{D}\left(\nabla^{\perp}\left((10+\Lmb)^{-\dlt}-I\right)\ttht^{\dlt} \cdot \nabla \tht^{D}\right) -  \int \tht^{D}\left(\nabla^{\perp} \left(\dlt\log(10 + \Lmb)  -I + \Lmb^{-\dlt}\right)\cdot\ttht^{\dlt}  \nabla \ttht^{\dlt}\right) \\
    \quad & =: L_1 + L_2 + L_3 +L_4.
\end{align*}
Using \eqref{eq:Tay-2} and the Sobolev embedding $H^{3}(\mathbb{R}^2) \hookrightarrow W^{1,\infty}(\mathbb{R}^2)$, we have
\begin{equation}\label{est: L1}
    |L_1| \le \normb{\tht^{D}}\normb{\nabla^{\perp}\left((10+\Lmb)^{-\dlt}-I\right)\tht^{D}}\normif{\nabla \ttht^{\dlt}}   \le C\dlt \nrmb{\tht^{D}}_{H^2}^2\nrmb{\ttht^{\dlt}}_{H^3}.
\end{equation}
Since $\nabla^{\perp}\left((10+\Lmb)^{-\dlt}-I\right)\tht^{\dlt}$ is divergence-free, integrating by parts gives
\begin{equation}\label{est: L23}
    L_2 + L_3 =  -\int \tht^{D}\left(\nabla^{\perp}\left((10+\Lmb)^{-\dlt}-I\right)\tht^{\dlt} \cdot \nabla \tht^{D}\right)=0.
\end{equation}
By \eqref{eq:Tay-1}, we obtain
\begin{equation}\label{est: L4}
\begin{split}
    |L_4|  &\le  \normb{\tht^{D}}\normb{\nabla^{\perp} \left(\dlt\log(10 + \Lmb)  -I + (10+\Lmb)^{-\dlt}\right)\ttht^{\dlt}} \normif{\nabla \ttht^{\dlt}} \le C \dlt^2 \normb{\tht^{D}}\nrmb{\ttht^{\dlt}}_{H^3}^2.
\end{split}
\end{equation} 
Next, we estimate $\int \Lmb^{s(t)}\tht^D \Lmb^{s(t)}\rd_t\tht^D$. We compute that
\begin{align*}
    \int \Lmb^{s}\tht^D \Lmb^{s}\rd_t\tht^D &= - \int\Lmb^{s} \tht^{D}\, \Lmb^{s}\left(\nabla^{\perp}\left((10+\Lmb)^{-\dlt}-I\right)\tht^{D} \cdot \nabla \ttht^{\dlt}\right) \\
    &\quad -\int \Lmb^{s}\tht^{D}\,\Lmb^{s}\left(\nabla^{\perp}\left((10+\Lmb)^{-\dlt}-I\right)\tht^{D} \cdot \nabla \tht^{D}\right) \\
    &\quad - \int \Lmb^{s}\tht^{D}\,\Lmb^{s}\left( \nabla^{\perp}\left((10+\Lmb)^{-\dlt}-I\right)\ttht^{\dlt} \cdot \nabla \tht^{D}\right) \\
    &\quad - \int \Lmb^{s}\tht^{D}\,\Lmb^{s}\left( \nabla^{\perp} \left(\dlt\log(10 + \Lmb) -I + (10+\Lmb)^{-\dlt}\right)\ttht^{\dlt} \cdot \nabla \ttht^{\dlt}\right) \\
    & =: H_1 + H_2 + H_3 + H_4.
\end{align*}
We claim that
\begin{equation}\label{est: H1}
\begin{split}
    |H_{1}|\le C\dlt \normhsps{\ttht^{\dlt}} \left(\nrmb{\log^{\frac{1}{2}}(10+\Lmb)\tht^D}^2_{H^{s}} 
    + \normhs{\tht^{D}}^2 \right) .
\end{split}
\end{equation}
As in \eqref{eq: decom-thmA}, we use \eqref{eq: decomp-xi} to decompose
\begin{equation*}
\begin{split}
    H_{1}
    &= -\Re \int  |\xi|^{ s} \overline{\widehat{\tht^{D}} }(\xi) \int  |\eta|^{ s}\left( (\xi-\eta)^{\perp}\left((10+|\xi-\eta|)^{-\dlt}-1 \right)\widehat{\tht^{D}}(\xi-\eta) \cdot i\eta  \widehat{\ttht^{\dlt}}(\eta)    \right)  \,\ud\eta  \, \ud\xi \\
    &\quad -\Re \int  |\xi|^{ s} \overline{\widehat{\tht^{D}} }(\xi) \int  |\xi-\eta|^{ s}\left( (\xi-\eta)^{\perp}\left((10+|\xi-\eta|)^{-\dlt}-1 \right)\widehat{\tht^{D}}(\xi-\eta) \cdot i\eta  \widehat{\ttht^{\dlt}}(\eta)    \right)  \,\ud\eta  \, \ud\xi \\
    &\quad -\Re \int  |\xi|^{ s} \overline{\widehat{\tht^{D}} }(\xi) \int  s(\xi-\eta)\cdot\eta |\xi-\eta|^{s-2}\\
    &\qquad\qquad\qquad\qquad\qquad\qquad \times \left((\xi-\eta)^{\perp}\left((10+|\xi-\eta|)^{-\dlt}-1 \right)\widehat{\tht^{D}}(\xi-\eta) \cdot i\eta  \widehat{\ttht^{\dlt}}(\eta)    \right)  \,\ud\eta  \, \ud\xi \\
    &\quad -\Re \int  |\xi|^{ s} \overline{\widehat{\tht^{D}} }(\xi) \int  R(\xi,\eta)\left( (\xi-\eta)^{\perp}\left((10+|\xi-\eta|)^{-\dlt}-1 \right)\widehat{\tht^{D}}(\xi-\eta) \cdot i\eta  \widehat{\ttht^{\dlt}}(\eta)  \right) \,\ud\eta  \, \ud\xi \\
    &=: H_{11}+H_{12}+H_{13}+H_{14}.
\end{split}
\end{equation*}
With \eqref{eq:Tay-2} and the Sobolev embedding $H^{s-2}(\bbR^2)\hookrightarrow L^{\infty}(\bbR^2) $, we estimate
\begin{equation*}
\begin{split}
|H_{11}| &\le \normb{\Lmb^{s} \tht^{D}}\normif{\nabla^{\perp}\left((10+\Lmb)^{-\dlt} -I\right)\tht^{D}} \normb{ \nabla\Lmb^s  \ttht^{\dlt}} \\
&\le C\normhs{\tht^{D}}\nrmb{\nabla^{\perp}\left((10+\Lmb)^{-\dlt} -I\right)\tht^{D}}_{H^{s-2}}\normhsps{\ttht^{\dlt}} \le C\dlt \normhsps{\ttht^{\dlt}}\normhs{\tht^{D}}^2.
\end{split}
\end{equation*}
We integrate by parts to decompose $H_{12}$ into
\begin{equation*}
\begin{split}
    H_{12} &= -\int \Lmb^{s} \tht^{D} \left(\nabla^{\perp}\Lmb^{s}\left(I-(10+\Lmb)^{-\dlt} \right)\tht^{D}\cdot  \nabla \ttht^{\dlt}\right) \\
    &=\int \Lmb^{s}\left(I-(10+\Lmb)^{-\dlt} \right)^{\frac{1}{2}}\tht^{D}\,\nabla^{\perp}\left(I-(10+\Lmb)^{-\dlt} \right)^{\frac{1}{2}}\left(\cdot  \nabla \ttht^{\dlt} \Lmb^{s} \tht^{D}\right)\\
    &= \int \Lmb^s\left(I-(10+\Lmb)^{-\dlt} \right)^{\frac{1}{2}}\tht^{D}\left[\nabla^{\perp}\left(I-(10+\Lmb)^{-\dlt} \right)^{\frac{1}{2}},\cdot\nabla \ttht^{\dlt}  \right] \Lmb^{s} \tht^{D} \\
    &\quad +\int \Lmb^s\left(I-(10+\Lmb)^{-\dlt} \right)^{\frac{1}{2}}\tht^{D} \, \nabla \ttht^{\dlt} \cdot  \nabla^{\perp}\Lmb^{s} \left(I-(10+\Lmb)^{-\dlt}\right)^{\frac{1}{2}}\tht^{D} \\
    &=H_{121}+H_{122}.
\end{split}
\end{equation*}
Using Lemma \ref{lem:Comm} and $\left|1-(10+|\xi|)^{-\dlt}\right|^{\frac{1}{2}} \le \dlt^{\frac{1}{2}} \log^{\frac{1}{2}}(10+|\xi|)$, we estimate
\begin{equation*}
\begin{split}
|H_{121}|
&\le \normb{\Lmb^s\left(I-(10+\Lmb)^{-\dlt} \right)^{\frac{1}{2}}\tht^{D}}
\normb{\left[\nabla^{\perp}\left(I-(10+\Lmb)^{-\dlt} \right)^{-\frac{1}{2}},\cdot\nabla \ttht^{\dlt}  \right] \Lmb^{s} \tht^{D}}\\
&\le C\dlt\normhs{\ttht^{\dlt}}\nrmb{\log^{\frac{1}{2}}(10+\Lmb)\tht^D}^2_{H^{s}}.
\end{split}
\end{equation*}
It follows from $\nabla^{\perp}\cdot \nabla \ttht^{\dlt}=0$ that $H_{122}=0.$ 
For $H_{13}$, we similarly further decompose
\begin{equation*}
\begin{split}
    H_{13} &= -s\int \Lmb^{s} \tht^{D}\left(\nabla\nabla^{\perp}\Lmb^{s-2}\left(I-(10+\Lmb)^{-\dlt} \right)\tht^{D}\cdot  \nabla^2 \ttht^{\dlt}\right) \\
    &= s\int \nabla\Lmb^{s-2} \left(I-(10+\Lmb)^{-\dlt} \right)^{\frac12}\tht^{D}\left[\nabla^{\perp}\left(I-(10+\Lmb)^{-\dlt} \right)^{\frac{1}{2}},\cdot\nabla^2 \ttht^{\dlt}  \right]\Lmb^s\tht^{D}  \\
    &\quad +s \int \nabla\Lmb^{s-2} \left(I-(10+\Lmb)^{-\dlt} \right)^{\frac12}\tht^{D} \, \nabla^2 \ttht^{\dlt} \cdot  \nabla^{\perp}\Lmb^{s} \left(I-(10+\Lmb)^{-\dlt}\right)^{\frac{1}{2}}\tht^{D} \\
    &=H_{131}+H_{132}.
\end{split}
\end{equation*}
With Lemma \ref{lem:Comm} and $\left|1-(10+|\xi|)^{-\dlt}\right|^{\frac{1}{2}} \le \dlt^{\frac{1}{2}} \log^{\frac{1}{2}}(10+|\xi|)$, we estimate
\begin{equation*}
\begin{split}
|H_{131}| 
&\le C\dlt\normhs{\ttht^{\dlt}}\nrmb{\log^{\frac{1}{2}}(10+\Lmb)\tht^D}^2_{H^{s}}.
\end{split}
\end{equation*}
The integration by parts yields
\begin{equation*}
    H_{132}=s \int \nabla^{\perp}\nabla\Lmb^{s-2} \left(I-(10+\Lmb)^{-\dlt} \right)^{\frac12}\tht^{D} \cdot \nabla^2 \ttht^{\dlt}   \Lmb^{s} \left(I-(10+\Lmb)^{-\dlt}\right)^{\frac{1}{2}}\tht^{D},
\end{equation*}
so that we have
\begin{equation*}
\begin{split}
|H_{132}| 
&\le \normb{\nabla^{\perp}\nabla\Lmb^{s-2} \left(I-(10+\Lmb)^{-\dlt} \right)^{\frac{1}{2}}\tht^{D}}\normif{\nabla^2 \ttht^{\dlt}}\normb{\Lmb^s\left(I-(10+\Lmb)^{-\dlt} \right)^{\frac{1}{2}}\tht^{D}}\\
&\le C\dlt\normhs{\ttht^{\dlt}}\nrmb{\log^{\frac{1}{2}}(10+\Lmb)\tht^D}^2_{H^{s}}.
\end{split}
\end{equation*}
Using $\left|(10+|\xi|)^{-\dlt}-1\right| \le \dlt \log(10+|\xi|)$, we can slightly modify the proof of the estimate of $IV$ in Theorem \ref{thm:wp-log} to obtain
\begin{equation*}
    |H_{14}| \le C\dlt\normhs{\ttht^{\dlt}}\normhs{\tht^{D}}^2.
\end{equation*}
Combining all the estimates, we have \eqref{est: H1}.  Next, we claim that
\begin{equation}\label{est: H2}
\begin{split}
    |H_{2}|\le C\dlt \normhs{\tht^{D}} \left(\nrmb{\log^{\frac{1}{2}}(10+\Lmb)\tht^D}^2_{H^{s}}
    + \normhs{\tht^{D}}^2 \right) .
\end{split}
\end{equation}
Using \eqref{eq: decomp-xi} again, we decompose
\begin{equation*}
\begin{split}
    H_{2}&= -\Re \int  |\xi|^{ s} \overline{\widehat{\tht^{D}} }(\xi) \int  |\eta|^{ s}\left( (\xi-\eta)^{\perp}\left((10+|\xi-\eta|)^{-\dlt}-1 \right)\widehat{\tht^{D}}(\xi-\eta) \cdot i\eta  \widehat{\tht^{D}}(\eta)    \right)  \,\ud\eta  \, \ud\xi \\
    &\quad -\Re \int  |\xi|^{ s} \overline{\widehat{\tht^{D}} }(\xi) \int  |\xi-\eta|^{ s}\left( (\xi-\eta)^{\perp}\left((10+|\xi-\eta|)^{-\dlt}-1 \right)\widehat{\tht^{D}}(\xi-\eta) \cdot i\eta  \widehat{\tht^{D}}(\eta)    \right)  \,\ud\eta  \, \ud\xi \\
    &\quad -\Re \int  |\xi|^{ s} \overline{\widehat{\tht^{D}} }(\xi) \int  s(\xi-\eta)\cdot\eta |\xi-\eta|^{s-2}\\
    &\qquad\qquad\qquad\qquad\qquad\qquad \times\left( (\xi-\eta)^{\perp}\left((10+|\xi-\eta|)^{-\dlt}-1 \right)\widehat{\tht^{D}}(\xi-\eta) \cdot i\eta  \widehat{\tht^{D}}(\eta)    \right)  \,\ud\eta  \, \ud\xi \\
    &\quad -\Re \int  |\xi|^{ s} \overline{\widehat{\tht^{D}} }(\xi) \int  R(\xi,\eta)\left( (\xi-\eta)^{\perp}\left((10+|\xi-\eta|)^{-\dlt}-1 \right)\widehat{\tht^{D}}(\xi-\eta) \cdot i\eta  \widehat{\tht^{D}}(\eta)  \right) \,\ud\eta  \, \ud\xi \\
    &=: H_{21}+H_{22}+H_{23}+H_{24}.
\end{split}
\end{equation*}
It is clear that $H_{21} =0$, and for the estimates of $H_{22}$, $H_{23}$, and $H_{24}$, we replace $\ttht^{\dlt}$ in the proofs of the estimates of $H_{12}$, $H_{13}$, and $H_{14}$ with $\tht^{D}$, respectively, which yields \eqref{est: H2}. Now we handle $H_3$.
Since $\nabla^{\perp}\left((10+\Lmb)^{-\dlt} -I\right)\ttht^{\dlt}$ is divergence-free, we have
\begin{equation*}
    H_{3}= 
    -\int \Lmb^{s}\tht^{D}\left[\Lmb^{s},\nabla^{\perp}\left((10+\Lmb)^{-\dlt} -I\right)\ttht^{\dlt} \cdot \right]\nabla\tht^{D},
\end{equation*}
and therefore, Lemma \ref{lem:K-P} and \eqref{eq:Tay-2} give
\begin{equation}\label{est: H3}
    \begin{split}
        |H_{3}| &\le \normb{\Lmb^{s}\tht^{D}}\normb{\left[\Lmb^{s},\nabla^{\perp}\left((10+\Lmb)^{-\dlt} -I\right)\ttht^{\dlt} \cdot \right]\nabla \tht^{D}}   \\
    &\le C\normhs{\tht^{D}}\left(\nrmb{\nabla^{\perp}\left((10+\Lmb)^{-\dlt} -I\right)\ttht^{\dlt}}_{H^{2+\eps}}
    \nrmb{\nb \tht^{D}}_{H^{s-1}} \right. \\
    &\qquad\qquad\qquad\left.+\normhs{\nabla^{\perp}\left((10+\Lmb)^{-\dlt} -I\right)\ttht^{\dlt}}
    \nrmb{\nb \tht^{D}}_{H^{1+\eps}}\right) \\
    &\le C\dlt\nrmb{\ttht^{\dlt}}_{H^{s+1+\eps}}\normhs{\tht^{D}}^2.
    \end{split}
\end{equation}
For $H_{4}$, noticing that $H^{s}(\bbR^2)$ is a Banach algebra, \eqref{eq:Tay-1} implies that
\begin{equation}\label{est: H4}
    \begin{split}
        |H_{4}| &\le \normb{\Lmb^{s}\tht^{D}}\normb{\Lmb^{s}\left(\nabla^{\perp} \left(\dlt\log(10 + \Lmb) -I + (10+\Lmb)^{-\dlt}\right)\ttht^{\dlt} \cdot \nabla \ttht^{\dlt}\right)}   \\
    &\le C \normhs{\tht^{D}}\normhs{\nabla^{\perp} \left(\dlt\log(10 + \Lmb) -I + (10+\Lmb)^{-\dlt}\right)\ttht^{\dlt}} 
    \normhs{\nabla \ttht^{\dlt}}  \\
    &\le C \dlt^2 \nrmb{\ttht^{\dlt}}_{H^{s+1+\eps}}^2
    \normhs{\tht^{D}}.
    \end{split}
\end{equation}
Summing from \eqref{est: L1} to \eqref{est: H4}, we have arrived at
\begin{equation} \label{est: Hs}
\begin{split}
    \frac{\ud}{\ud t} \nrmb{\tht^D}^2_{H^{s}}
    &\le 2\left(\dot{s}+ C\dlt\left(\nrmb{\ttht^{\dlt}}_{H^{s+1+\eps}}+\normhs{\tht^{D}}\right)\right) \nrmb{\log^{\frac{1}{2}}(10+\Lmb)\tht^D}^2_{H^{s}} \\
    &\quad +C\dlt\left(\nrmb{\ttht^{\dlt}}_{H^{s+1+\eps}}\normhs{\tht^{D}}^2 + \normhs{\tht^{D}}^3 + \dlt
    \nrmb{\ttht^{\dlt}}_{H^{s+1+\eps}}^2
    \normhs{\tht^{D}}\right).
\end{split}
\end{equation}
We now define
\begin{equation*}\label{s(t)}
    s(t):=s_0 - C\dlt\left(1+ \sup_{t\in\left[0,T^{log}\right]}\nrmb{\tht (t) }_{H^{s+1+\eps}}\right)t,
\end{equation*}
and
\begin{equation*}\label{C^*}
    C^*:=\frac{\min\left\{ {1,\,T^{log}},\,s_0-4\right\}}{100C\left(1+\sup_{t\in[0, {T^{log}}]}\nrmb{\tht(t)}_{H^{s+1+\eps}}\right)^2}
\end{equation*}
We claim that for any $\dlt \in (0,1)$,
\begin{equation}\label{est: s(t)}
     s(t)>4,
\end{equation}
and
\begin{equation}\label{est: tht^D}
    \nrmb{\tht^D(t)}_{H^{s(t)}}\le \dlt
\end{equation}
on $\left[0,C^*\dlt^{-1}\right]$.
We can check \eqref{est: s(t)} with the definitions of $s(t)$ and $C^*$.
To obtain \eqref{est: tht^D},
suppose, on the contrary, that
\begin{equation*}\label{Tbar}
    \bar{T}:=\inf{\left\{t\,:\;\nrmb{\tht^D(t)}_{H^{s(t)}}> \dlt \right\}}< C^*\dlt^{-1}.
\end{equation*}
Then since 
\begin{equation*}
    \dot{s}+ C\dlt\left(\nrmb{\ttht^{\dlt}}_{H^{s+1+\eps}}+\normhs{\tht^{D}}\right) \le 0
\end{equation*}
on $[0,\bar{T}]$,
\eqref{est: Hs} implies that
\begin{equation*}
    \frac{\ud}{\ud t} \nrmb{\tht^D}_{H^{s}} \le C\dlt\left(\nrmb{\ttht^{\dlt}}_{H^{s+1+\eps}}\normhs{\tht^{D}} + \normhs{\tht^{D}}^2 + \dlt 
    \nrmb{\ttht^{\dlt}}_{H^{s+1+\eps}}^2\right)
\end{equation*}
on $[0,\bar{T}]$.
In order to make use of Lemma \ref{lem:PC}, we multiply this by $\exp{\left(-C\dlt\int_{0}^{t} \nrmb{\ttht^{\dlt}(\tau)}_{H^{s+1+\eps}}d\tau\right)}$ and consider the quantity
\begin{equation*}
    y(t):= \normhs{\tht^{D}(t)}\exp{\left(-C\dlt\int_{0}^{t} \nrmb{\ttht^{\dlt}(\tau)}_{H^{s+1+\eps}}d\tau\right)}.
\end{equation*}
Denoting
\begin{equation*}
    G:=C\dlt\exp{\left(C\dlt\int_{0}^{\bar{T}} \nrmb{\ttht^{\dlt}(\tau)}_{H^{s+1+\eps}}d\tau\right)},
\end{equation*}
we then obtain
\begin{equation*}
    \left\{
    \begin{aligned}
    &\frac{\ud}{\ud t}y(t) \le C \dlt^2 \nrmb{\ttht^{\dlt} (t) }_{H^{s+1+\eps}}^2  + Gy^2(t), \\
    &y(0)=0
    \end{aligned}
    \right.
\end{equation*}
on $[0,\bar{T}]$.
Replacing $\nu$, $T$, and $F(t)$ in Lemma \ref{lem:PC} by $C\dlt^2$, $\bar{T}$, and $\nrmb{\ttht^{\dlt} (t) }_{H^{s+1+\eps}}^2$, respectively, we have
\begin{equation*}
\begin{split}
    8\nu TG &\int_0^{T}F(t)dt \le 8\dlt \left(C\dlt\bar{T}\sup_{t\in\left[0,T^{log}\right]}\nrmb{\tht (t) }_{H^{s+1+\eps}}\right)^2 \exp{\left(C\dlt\bar{T}\sup_{t\in\left[0,T^{log}\right]}\nrmb{\tht (t) }_{H^{s+1+\eps}}\right)} \\
    &< 8\dlt \left(CC^* \sup_{t\in\left[0,T^{log}\right]}\nrmb{\tht (t) }_{H^{s+1+\eps}}\right)^2 \exp{\left(CC^* \sup_{t\in\left[0,T^{log}\right]}\nrmb{\tht (t) }_{H^{s+1+\eps}}\right)} \le 1
\end{split}
\end{equation*}
for any $\dlt\in(0,1)$.
Then Lemma \ref{lem:PC} implies that
\begin{equation*}
\begin{split}
    \normhs{\tht^{D}(t)} 
    &\le 12C\dlt^2 \bar{T} \left(\sup_{t\in\left[0,T^{log}\right]}\nrmb{\tht (t) }_{H^{s+1+\eps}}\right)^2 \exp{\left(C\dlt\bar{T}\sup_{t\in\left[0,T^{log}\right]}\nrmb{\tht (t) }_{H^{s+1+\eps}}\right)} \le \frac{\dlt}{2}
\end{split}
\end{equation*}
for any $\dlt \in (0,1)$ and $t \in \left[0,\bar{T}\right]$, which is a contradiction to the definition of $\bar{T}$.
\end{proof}

\begin{proof}[Proof of Theorem \ref{thm:gwp}]
By the local well-posedness theory of \eqref{eq:diss-sqg-delta}, there exist a maximal existence time $T_{max}^{\dlt}>0$ and a unique solution $\tht^{\dlt}\in L^{\infty}([0,T_{max}^{\dlt});H^s)\cap L^{2}([0,T_{max}^{\dlt});\Psi^{-1}H^s)$ to \eqref{eq:diss-sqg-delta}, where $\Psi^{-1}H^s$ is a function space defined by
\begin{equation*}
    \Psi^{-1}H^s:=\left\{f: \Psi f \in H^s \right\}.
\end{equation*}
In the same manner as the proof of \eqref{est: H2}, we estimate
\begin{equation*}
    \begin{split}
        \frac{1}{2}\frac{\ud}{\ud t} \normhs{\tht^\dlt}^2 + \kpp \normhs{\Psi \tht}^2
        &\le C\dlt\normhs{\tht^{\dlt}}\left(\normhs{\log(10+\Lmb)\tht^{\dlt}}^2+\normhs{\tht^{\dlt}}^2\right)\\
        &\le C\dlt\normhs{\tht^{\dlt}}\normhs{\log(10+\Lmb)\tht^{\dlt}}^2
    \end{split}
\end{equation*}
on $[0,T_{max}^{\dlt})$.
Defining $\dlt^*:=\frac{\kpp}{10C\normhs{\tht_0}},$ 
the assumption \eqref{eq:diss-sqg-cond} implies that for any $\dlt \in (0,\dlt^*)$,
\begin{equation*}
    \begin{split}
        \frac{1}{2}\frac{\ud}{\ud t} \normhs{\tht^\dlt}^2
        &\le \left(C\dlt\normhs{\tht^{\dlt}}-\kpp\right)\normhs{\log(10+\Lmb)\tht^{\dlt}}^2  \le \kpp \left(\frac{\normhs{\tht^{\dlt}}}{10\normhs{\tht_0}}-1 \right)\normhs{\log(10+\Lmb)\tht^{\dlt}}^2
    \end{split}
\end{equation*}
on $[0,T_{max}^{\dlt})$. To obtain $T_{max}^{\dlt}=\infty$ for any $\dlt \in (0,\dlt^*)$, it suffices to show that
\begin{equation*}
    \Bar{T}:=\inf{\left\{t\,:\;\nrmb{\tht^{\dlt}(t)}_{H^{s}}> 2\normhs{\tht_0} \right\}}\ge T_{max}^{\dlt}
\end{equation*}
for any $\dlt \in (0,\dlt^*)$.
Suppose, on the contrary, that $\Bar{T}<T_{max}^{\dlt}$. Then on $[0,\Bar{T}]$, we have
\begin{equation*}
    \frac{1}{2}\frac{\ud}{\ud t} \normhs{\tht^\dlt}^2 \le -\frac{4\kpp}{5}\normhs{\log(10+\Lmb)\tht^{\dlt}}^2 \le 0,
\end{equation*}
and therefore, 
\begin{equation*}
    \normhs{\tht^{\dlt}(\Bar{T})} \le \normhs{\tht_0},
\end{equation*}
which is a contradiction to the definition of $\Bar{T}$.
As in the case of \eqref{eq:sqg-delta}, we can check that $\tht^{\dlt}\in L^{\infty}_t H^{s}_x$ with $s>3$ implies $\tht^{\dlt}\in L^{\infty}_t H^{\infty}_x$ (see \cite{CCCGW}).
\end{proof}

\section*{Declarations}

\subsection*{Acknowledgments}{D.~Chae was supported partially  by NRF grant 2021R1A2C1003234. I.-J.~Jeong has been supported by the NRF grant funded by the Korea government(MSIT)(grant No. 2022R1C1C1011051, RS-2024-00406821). S.-J.~Oh was supported by the Samsung Science and Technology Foundation under Project Number SSTF-BA1702-02, a Sloan Research Fellowship and a National Science Foundation CAREER Grant under NSF-DMS-1945615.}

\subsection*{Data availability statement} The authors declare that data sharing is not applicable to this article as no datasets were generated or analyzed during the current study.

\subsection*{Conflict of Interests/Competing Interests} The authors have no relevant conflict of interests and competing interests to disclose that are relevant to the content of this article.

\bibliographystyle{amsplain}
\providecommand{\bysame}{\leavevmode\hbox to3em{\hrulefill}\thinspace}
\providecommand{\MR}{\relax\ifhmode\unskip\space\fi MR }
\providecommand{\MRhref}[2]{%
	\href{http://www.ams.org/mathscinet-getitem?mr=#1}{#2}
}
\providecommand{\href}[2]{#2}


\end{document}